\documentclass{jcmlatex}
\def\JCMvol{xx}
\def\JCMno{x}
\def\JCMyear{2022}
\def\JCMreceived{January 22, 2020}
\def\JCMrevised{June 9, 2021}
\def\JCMaccepted{December 31, 2021}

\usepackage{color}
\usepackage{algorithm}
\usepackage{algorithmic}
\graphicspath{{figure/}}
\usepackage{lscape}
\usepackage{url,cases}
\usepackage{subfigure}
\usepackage{dashrule}
\usepackage{arydshln}
\usepackage{xfrac}

\numberwithin{equation}{section}

\def\be{\begin{equation}}
\def\ee{\end{equation}}

\def\X{\mathcal{X}}
\def\Y{{\cal Y}}

\def\la{\lambda}
\def \[{\begin{equation}}
\def \]{\end{equation}}

\def\R{{\mathbb R}}

\def\L{{\cal L}}

\def\S{\mathcal{U}}
\def\D{{\cal D}}

\def\st{\mathrm{s. t.}}

\def\nn{\nonumber}
\def\Tsf{\mathsf{T}}

\def\llang{\left\langle}
\def\rrang{\right\rangle}
\def\Sigt{\Sigma_{1}}
\def\Sigtt{\Sigma_{2}}
\def\SigT{\Sigma}
\def\Pone{S}
\def\Ptwo{T}

\def\Ghat{\widehat{G}}
\newcommand{\dom}{\mathbb{D}}
\newcommand{\rev}[1]{{\color{black}{#1}}} 
\newcommand{\revise}[1]{{\color{black}{#1}}} 
\newcommand{\revi}[1]{{\color{black}{#1}}} 
\newcommand{\revis}[1]{{\color{black}{#1}}} 
\begin{document}

\markboth{Y.~GU, B.~JIANG AND D.R.~HAN}{An Indefinite-proximal-based Strictly Contractive Peaceman-Rachford Splitting Method}

\title{AN INDEFINITE-PROXIMAL-BASED STRICTLY CONTRACTIVE PEACEMAN-RACHFORD SPLITTING METHOD}

\author{Yan Gu
\thanks{Department of Mathematics, Nanjing University of Aeronautics and Astronautics,\\ Nanjing, 210016, China \\ Key Laboratory of Mathematical Modelling and High Performance Computing of Air Vehicles (NUAA), MIIT, Nanjing, 210016, China \\ Email: guyanmath@nuaa.edu.cn}
\and
 Bo Jiang
\thanks{School of Mathematical Sciences, Key Laboratory for NSLSCS of Jiangsu Province, \\Nanjing
Normal University, Nanjing, 210023, China\\ Email: jiangbo@njnu.edu.cn}
\and
Deren Han\footnote{Corresponding author}
\thanks{LMIB, School of Mathematics and Systems Science,  Beihang University, Beijing, 100191, China \\ Key Laboratory for NSLSCS of Jiangsu Province, China\\ Email: handr@buaa.edu.cn}
}

\maketitle

\begin{abstract}
The Peaceman-Rachford splitting method is efficient for minimizing \rev{a convex optimization problem with a \revise{separable} objective function and linear constraints.} However, its convergence was not guaranteed without extra requirements. He {\it et al.}  (SIAM J. Optim. 24: 1011 - 1040, 2014) proved the convergence of a strictly contractive Peaceman-Rachford splitting method by employing a suitable underdetermined relaxation factor.  In this paper, we further extend the so-called strictly contractive Peaceman-Rachford splitting method by using two different relaxation factors. Besides, \rev{motivated by the recent advances on the ADMM type method with indefinite proximal terms, we employ the indefinite proximal term in the strictly contractive Peaceman-Rachford splitting method.}
We show that the proposed indefinite-proximal strictly contractive Peaceman-Rachford splitting method is convergent and also prove the $o(1/t)$ convergence rate in the nonergodic sense. The numerical tests on the $l_1$ regularized least square problem demonstrate the efficiency of the proposed method.
\end{abstract}



\section{Introduction}
\label{sec:into}
We consider the convex minimization problem with linear constraints and a separable objective function:
\[\label{prob:cp}
 \min \,  \theta_1(x) +\theta_2(y),  \quad \st \quad   Ax+By=b, \; x \in \X, \; y \in \Y,
\]
\revi{where $A\in\R^{m\times {n_1}}$ and $B\in\R^{m\times {n_2}}$, $b\in\R^m$, $\X:= \R^{n_1}$, $\Y := \R^{n_2}$.    The functions  $\theta_1(x) := p(x) + f(x)$ and $\theta_2(y) := h(y) + g(y)$, where $p:\X \to (-\infty, +\infty]$ and $h:\Y \to (-\infty, +\infty]$ are proper closed convex (could be nonsmooth) functions; $f:\X \to (-\infty, +\infty)$ and $g:\Y \to (-\infty, +\infty)$ are two convex functions with Lipschitz continuous gradients on $\X$ and $\Y$.  Throughout, the solution set of (\ref{prob:cp}) is assumed to be nonempty. Note that one can also consider $\X$ and $\Y$ as general real finite dimensional Euclidean or Hilbert  spaces, see \cite{chen2021equivalence,han2017linear,zhang2017linearly} for instances. For ease of presentation, we adopt the $\X$ and $\Y$ as the ordinary $\R^{n_1}$ and $\R^{n_2}$ in this paper.
}

Let  $\L_{\beta}(x,y,\la)$ be the augmented Lagrangian function for \eqref{prob:cp} that defined by
\[\label{augL}
\L_{\beta}(x,y,\la):=\theta_1(x)+\theta_2(y)-\langle\la,Ax+By-b\rangle+\frac{\beta}{2}\|Ax+By-b\|^2,\]
in which $\la\in \R^m$ is the multiplier associated to the linear constraint and  $\beta>0$ is a penalty parameter.

A well-known method \revise{called} alternating direction method of multipliers (ADMM) is efficient to \revise{minimize} such problems. It was observed in Gabay and Mercier \cite{gabay1975dual}, Glowinski and Marrocco \cite{glowinski1975approximation} that ADMM can be derived from applying the Douglas-Rachford operator splitting method \cite{douglas1956numerical} to the dual of \revise{the problem} \eqref{prob:cp}. The iterative sequence is given as the following recursion:
\begin{subnumcases}{}
x^{k+1}=\arg\min_{x\in\X}\L_{\beta}(x,y^k,\la^k),\label{alx}\\
y^{k+1}=\arg\min_{y\in\Y}\L_{\beta}(x^{k+1},y,\la^k),\label{aly}\\
\la^{k+1}=\la^k-\beta(Ax^{k+1}+By^{k+1}-b).\label{all}
\end{subnumcases}
Based on another \rev{classical} operator splitting method, i.e., \revise{the} Peaceman-Rachford operator splitting method \cite{peaceman1955numerical}, one can derive the following similar method for \eqref{prob:cp}:
\begin{subnumcases}{}
x^{k+1}=\arg\min_{x\in\X}\L_{\beta}(x,y^k,\la^k),\label{alxp}\\
\la^{k+\frac{1}{2}}=\la^k-\beta(Ax^{k+1}+By^{k}-b),\label{allb}\\
y^{k+1}=\arg\min_{y\in\Y}\L_{\beta}(x^{k+1},y,\la^{k+\frac{1}{2}}),\label{alyp}\\
\la^{k+1}=\la^{k+\frac{1}{2}}-\beta(Ax^{k+1}+By^{k+1}-b).\label{allp}
\end{subnumcases}
While the global convergence of the alternating direction method of multipliers \eqref{alx}-\eqref{all} can be established under very mild conditions \cite{boyd2011distributed}, the convergence of the Peaceman-Rachford-based method \eqref{alxp}-\eqref{allp} can not be guaranteed without further conditions \cite{corman2014generalized}.

{He et al. \cite{he2014strictly} proposed a modification of \eqref{alxp}-\eqref{allp} by introducing a parameter $\alpha$ to the update scheme of the dual variable $\lambda$ in \eqref{allb} and
\eqref{allp}. Note that when $\alpha=1$, it is the same as \eqref{alxp}-\eqref{allp}. They explained the \revise{non-convergence} behavior of  \eqref{alxp}-\eqref{allp} from the contractive perspective, i.e., the distance from the iterative point to the solution set is merely nonexpansive, but not contractive. Under the condition that $\alpha\in (0,1)$, they proved the same sublinear convergence rate as that for ADMM \cite{he2012convergence}. Particularly, they showed that it achieves an approximate solution of \eqref{prob:cp} with the accuracy of $O(1/t)$ after $t$ iterations\footnote{A worst-case $O(1/t)$ convergence rate means the accuracy to a solution under certain criteria is of the order $O(1/t)$ after $t$ iterations of an iterative scheme; or equivalently, it requires at most $O(1/\epsilon)$ iterations to achieve an approximate solution with an accuracy of $\epsilon$. See, e.g., \cite{nesterov1983method, nesterov2013gradient}.}, both in the ergodic and  nonergodic sense.
Besides, \revis{Gu et al. \cite{gu2015semi}}  and  He et al. \cite{he2016convergence} took two different constants $\alpha$ and $\gamma$ to different step sizes in \eqref{allb} and \eqref{allp}. The convergence results, including global convergence, the worst-case $O(1/t)$ convergence rate in the ergodic sense, have been established in it but without the worst-case $O(1/t)$ convergence rate in the nonergodic sense.  Chen et al. \cite{Chen2016} proposed a variant Peaceman-Rachford splitting method in a prediction-correction framework.
\revise{\rev{For some recent advances of the Peaceman-Rachford splitting method}, one can refer to \cite{li2017peaceman, wu2017symmetric, bai2018generalized, he2018relaxed, gao2018symmetric, jiang2018generalized}, \rev{to name a few.}}

\rev{Considering that in many cases the subproblem in \eqref{alx}-\eqref{all} and \eqref{alxp}-\eqref{allp} might  be difficult to solve and that in some applications \revise{$\theta_1$} or $\theta_2$ is a convex quadratic function,} Eckstein \cite{eckstein1994some} and He et al. \cite{HLHY2002}  considered to add proximal terms to the subproblems for different purpose. Fazel et al. \cite{fazel2013hankel} proposed the following semi-proximal ADMM scheme:
\begin{subnumcases}{\label{equ:sADMM}}
 x^{k+1}=\arg\min_{x\in\X}\ \L_{\beta}(x,y^k,\la^k) + \frac12 \|x - x^{k}\|_{\Pone}^2,\label{equ:sADMM1}\\
y^{k+1}=\arg\min_{y\in\Y}\ \L_{\beta}(x^{k+1},y,\la^k) + \frac12 \|y - y^{k}\|_{\Ptwo}^2,\label{equ:sADMM2} \\
 \la^{k+1}=\la^{k}-\gamma\beta(Ax^{k+1}+By^{k+1}-b).\label{equ:sADMM3}
\end{subnumcases}
where $\gamma \in (0, (1+\sqrt 5)/2)$. \rev{They} allowed $S$ and $T$ to be positive \revise{semidefinite} which makes the algorithm more flexible. We refer the reader to \cite{deng2016global, he2012convergence, fazel2013hankel, xu2011class} for a brief history of the development of the semi-proximal ADMM and the corresponding convergence results.
\rev{To further relax the requirements of the proximal terms,
Li et al. \cite{li2016majorized} considered majorized ADMM with indefinite proximal $S$ and $T$. They established the convergence and the sublinear convergence rate under some mild assumptions. The numerical results in \cite{li2016majorized} showed that the (majorized) ADMM with indefinite proximal term always performs better than that with semidefinite proximal terms.} Very recently, \revise{He et al. \cite{he2019optimally} obtained a linearized ADMM with \rev{an optimal indefinite proximal term. In their method, $S = 0$ and}
$T$ in \eqref{equ:sADMM2} is chosen by
\[\label{T0}
T = \tau r I_{n_2} - \beta B^{\Tsf} B ~~\mathrm{with}~~ r > \beta\|B^{\Tsf} B\|,~~\tau \in (0.75,1),
\]
\revi{where $\|B^{\Tsf} B\|$ means the spectral norm of $B^{\Tsf}B$.}
\rev{Note that they require that the dual stepsize  $\gamma = 1$ in \eqref{equ:sADMM}}. The small value $\tau \in (0.75,1)$ can ensure the proximal term has less weight for the $y$-subproblem \eqref{equ:sADMM2}, and thus allows for \rev{a larger step.}}
Solving a general problem, i.e., finding zeros of a maximal operator, using a proximal point algorithm with indefinite proximal term, was recently developed in \cite{Jiang2021Indefinite}.

{It is natural to extend the proximal ADMM to the proximal Peaceman-Rachford splitting method}. For convenience, we first introduce the whole update scheme of the {\it indefinite-proximal-based strictly contractive Peaceman-Rachford splitting method (iPSPR)}
\begin{subnumcases}{\label{equ:sP-PRSM}}
 x^{k+1}=\arg\min_{x\in\X}\ \L_{\beta}(x,y^k,\la^k) + \frac12 \|x - x^{k}\|_{\Pone}^2,\label{equ:sP-PRSM1}\\
 \la^{k+\frac{1}{2}}=\la^k-\alpha\beta(Ax^{k+1}+By^{k}-b),\label{equ:sP-PRSM2}\\
 y^{k+1}=\arg\min_{y\in\Y}\ \L_{\beta}(x^{k+1},y,\la^{k+\frac{1}{2}}) + \frac12 \|y - y^{k}\|_{\Ptwo}^2,\label{equ:sP-PRSM3}\\
 \la^{k+1}=\la^{k+\frac{1}{2}}-\gamma\beta(Ax^{k+1}+By^{k+1}-b).\label{equ:sP-PRSM4}
\end{subnumcases}
where \rev{$S$ and $T$ are symmetric and possibly indefinite.}
Gao et al. \cite{gao2018symmetric} \rev{considered the generalized ADMM with indefinite proximal term, which corresponds to \eqref{equ:sP-PRSM} with $S = 0$ and $\gamma = 1$. The proximal term $T$ takes a similar formulation as \eqref{T0} but with} $\tau \in [ \frac{\alpha^2 - \alpha + 4}{\alpha^2 - 2\alpha + 5},1).$ Jiang et al. \cite{jiang2018generalized} \rev{considered the same generalized ADMM as in \cite{gao2018symmetric}, but they give an optimal bound of $\tau$ as} $\tau \in (\frac{3 + \alpha}{4}, 1)$. \rev{For other related \rev{works} one can refer to  \cite{li2015proximal, sun2017symmetric}.}

\rev{In this paper, we focus on \eqref{equ:sP-PRSM} with indefinite $S$ and $T$.}
\revise{Our main contributions are \rev{two-fold}.
Firstly,} motivated by the nice analysis techniques in \cite{he2014strictly} and \cite{xu2011class},  \rev{we prove the global convergence of iPSPR under some assumptions on $S$ and $T$, see \eqref{equ:condition:S} and \eqref{equ:condition:T}, in which the stepsizes $\alpha$ and $\gamma$ are in the range
\begin{align}\label{stepsize}
(\alpha, \gamma) \in  \dom := \left\{(\alpha, \gamma) : 0\leq \alpha < 1, 0\leq \gamma <  \frac{1 - \alpha + \sqrt{(1+\alpha)^2 + 4(1 - \alpha^2)}}{2}, \alpha + \gamma > 0\right\}.
\end{align}
\revise{With some additional mild requirements, see \eqref{equ:T:nonergodic}, we \rev{prove that  the  iPSPR} is $o(1/t)$ sublinearly convergent in the nonergodic sense.}} \revise{Secondly,  our proposed requirements on the  proximal $T$ can cover \rev{some} existing results, such as the special linearized choice \eqref{T0} in \cite{he2019optimally, jiang2018generalized}}. \rev{More importantly, our proposed requirements on the  proximal $T$ employs both the Hessian information of the objective function and the information of $\beta B^{\Tsf}B$ for the first time. Note that He et al. \cite{he2019optimally} only uses the information of $\beta B^{\Tsf}B$, while Li et al.  \cite{li2016majorized} only considers the Hessian information of the objective function.}

The rest of this paper is organized as follows. In section  \ref{section:preliminaries}, we give the optimality condition of \eqref{prob:cp} by using the variational inequality and also list some assertions which will be used in later analysis. In section \ref{section:convergence}, we first give the contraction analysis of iPSPR \eqref{equ:sP-PRSM}, \revise{and then establish the global convergence. We will discuss how to choose $T$ in the end of section \ref{section:convergence}. The detailed formulae will be given for the different ranges of the parameters $\alpha$ and $\gamma$. We discuss the nonergodic sublinear convergence rate in section \ref{section:sublinear}.}  In section \ref{section:numerical}, we test the $l_1$ regularized least square problem to show the efficiency of the proposed iPSPR \eqref{equ:sP-PRSM}. Finally, we make some conclusions in section \ref{section:conclusions}.
}

\section{Preliminaries}\label{section:preliminaries}

In this section, we give the optimality condition of \eqref{prob:cp} and  some notations or relations which will be frequently used in our analysis. Denote $\Omega = \X \times \Y \times \R^m$. Let $\S$ be the feasible set of \eqref{prob:cp}, namely, $\S = \{(x,y) : Ax + By = b, x \in \X, y \in \Y\}$ and denote $\D = \S \times \R^m$. Throughout this paper, we make the following assumption.
\begin{assumption}\label{ass:opt} Let $\Omega^*\subset \D$ be the set whose elements are the  optimal solutions  of \eqref{prob:cp} and the associating dual solutions of \eqref{prob:cp}. Throughout the paper, we assume that
$\Omega^*$ is non-empty.
\end{assumption}

\subsection{Optimality condition of \eqref{prob:cp}}

Owing to the convexity of $\theta_1(\cdot) $ and $\theta_2(\cdot)$,  there exist two positive semidefinite matrices  $\Sigt$ and $\Sigtt$ such that for all $x, \; x' \in \R^{n_1}$ and \rev{$\xi_x \in \partial \theta_1(x)$, $\xi_x' \in \partial \theta_1(x')$,}
\[ \label{equ:x:convex}
  \langle x - x',  \xi_x - \xi_{x}' \rangle  \geq \|x - x'\|_{\Sigt}^2,
\]
and for all $y,\; y' \in \R^{n_2}$, \rev{$\xi_y \in \partial \theta_2(y)$, $\xi_y' \in \partial \theta_2(y')$,}
\[\label{equ:y:convex}
\langle y - y', \xi_y - \xi_y' \rangle \geq \|y - y' \|_{\Sigtt}^2.
\]

\rev{Denote  $u = \begin{pmatrix}x\\y\end{pmatrix},\; v =  \begin{pmatrix}y\\\lambda\end{pmatrix}$ and  $w = \begin{pmatrix}x\\ y\\\lambda\end{pmatrix}$. For given $w$,  and some specific subgradients $\xi_x \in \partial \theta_1(x)$ and $\xi_y \in \partial \theta_2(x)$, we define
$ F(w, \xi_x, \xi_y) = \begin{pmatrix}
\xi_x -A^{\Tsf} \lambda \\ \xi_y -B^{\Tsf} \lambda \\ Ax + By - b \end{pmatrix}.$}  Due to the convexity of $\theta_1(\cdot)$ and $\theta_2(\cdot)$,  it is easy to show that the operator $F(\cdot)$ is \rev{monotone}. Specifically, for any \rev{$w, w' \in \D$}, we have
\begin{equation}\label{equ:F:mono}
\langle w - w', \rev{F(w, \xi_x, \xi_y) - F(w',\xi_x', \xi_y')} \rangle = \left \langle \begin{pmatrix} x - x' \\ y - y' \end{pmatrix},  \begin{pmatrix} \xi_x - \xi_x' \\ \xi_y - \xi_y' \end{pmatrix}  \right\rangle \geq \|u - u'\|_{\SigT}^2,
\end{equation}
where $\SigT = \begin{pmatrix} \Sigt & 0 \\ 0 & \Sigtt \end{pmatrix}$
and  the inequality is due to \eqref{equ:x:convex} and \eqref{equ:y:convex}.

\rev{Following Theorem 3.1.24 in \cite{nesterov2018lectures}}, we
say that $w^* \in \Omega^*$  \rev{if  and only if there exists $\xi_{x}^* \in \partial \theta_1(x^*)$ and $\xi_{y}^* \in \partial \theta_2(y^*)$ such that $\langle x - x^*, \xi_{x}^*\rangle + \langle y - y^*, \xi_{y}^*\rangle \geq 0$, which is further equivalent to}
\[\label{equ:opt:cond:D}
 \llang w - w^*, \rev{F(w^*, \xi_x^*, \xi_y^*)} \rrang \geq 0, \quad \forall w \in \D,
\]
\rev{because  $\llang w - w^*, F(w^*, \xi_x^*, \xi_y^*) \rrang = \langle x - x^*, \xi_{x}^*\rangle + \langle y - y^*, \xi_{y}^*\rangle \geq 0$, $ \forall w \in \D$.
}

\subsection{Some notations}
We use the symbol $0$ to denote a zero matrix, whose size can be always easily identified   form the context. We use $\|\cdot \|$ to denote the 2-norm of a vector. We denote  $\|z\|_G^2 = z^{\Tsf} Gz$  for $z \in \R^n$ and $G \in \R^{n \times n}$. For a real symmetric matrix $Z$, we mark $Z\succeq 0$ (\rev{resp. $Z \succ 0$}) if $Z$ is positive semidefinite (\rev{resp. positive definite}). For any given symmetric matrix $T$, we decompose it as
$$
T = T_{+} - T_{{-}} \quad \mbox{with}\quad \ T_{+} \succeq 0\quad \mbox{and}\quad  \ T_{-} \succeq 0.
$$

To make the analysis more \rev{elegantly}, we use $r^k = Ax^k + By^k - b$ for short.  Similarly, for any $w \in \Omega$, we denote $r(w) = Ax + By - b$. Obviously, there holds that $r(w) = 0$ for any $w \in \D$.  For ease of the analysis, we define\
\begin{align} \label{equ:H}
  H =\frac{1}{\alpha + \gamma} \begin{pmatrix} (\alpha + \gamma - \alpha \gamma) \beta B^{\Tsf} B & -\alpha B^{\Tsf} \\ -\alpha B & \frac{1}{\beta} I_m\end{pmatrix} \quad
\end{align}
and
\begin{align}\label{equ:H:hat}
\hat H :=
  \left(
  \begin{matrix}
  T + \Sigma_2 & 0 \\ 0 & 0
  \end{matrix}
   \right)
  + H =
  \rev{\begin{pmatrix} \displaystyle
   {\Ptwo} + \Sigtt + \frac{\alpha + \gamma - \alpha \gamma}{\alpha + \gamma} \beta B^{\Tsf} B & \displaystyle -\frac{\alpha}{\alpha + \gamma} B^{\Tsf} \\
 \displaystyle  -\frac{\alpha}{\alpha + \gamma} B & \displaystyle \frac{1}{(\alpha + \gamma) \beta} I_m
  \end{pmatrix}.
  }
\end{align}
Denote $P = \begin{pmatrix} {\Pone} & 0 \\ 0 & {\Ptwo}\end{pmatrix}$  and define
 \begin{align} \label{equ:G}
  G := \begin{pmatrix} P & 0 \\ 0 & 0  \end{pmatrix} + \begin{pmatrix} 0 & 0 \\ 0 & H \end{pmatrix}=
   \begin{pmatrix} {\Pone}  &  0 & 0 \\
  0 &  \displaystyle {\Ptwo} + \frac{\alpha + \gamma - \alpha \gamma}{\alpha + \gamma} \beta B^{\Tsf} B &\displaystyle  -\frac{\alpha}{\alpha + \gamma} B^{\Tsf} \\
  0  & \displaystyle -\frac{\alpha}{\alpha + \gamma} B & \displaystyle \frac{1}{(\alpha + \gamma) \beta} I_m
  \end{pmatrix}
\end{align}
and
\begin{align}\label{equ:Ghat}
\Ghat := \begin{pmatrix}  \Sigma & 0 \\ 0 & 0  \end{pmatrix} + G  =  \begin{pmatrix}  {\Pone} + \Sigt   &  0 & 0 \\
  0 & \displaystyle  {\Ptwo} + \Sigtt + \frac{\alpha + \gamma - \alpha \gamma}{\alpha + \gamma} \beta B^{\Tsf} B &\displaystyle  -\frac{\alpha}{\alpha + \gamma} B^{\Tsf} \\
  0  &\displaystyle  -\frac{\alpha}{\alpha + \gamma} B & \frac{1}{(\alpha + \gamma) \beta} I_m
  \end{pmatrix}
  = \begin{pmatrix} \Pone + \Sigma_1  & 0 \\ 0 & \hat H \end{pmatrix}.
\end{align}
\rev{It follows from \eqref{equ:G} and \eqref{equ:Ghat} that} for any $w, w' \in \Omega$
\[ \label{equ:GH_norm}
\| w - w'\|_G^2 = \|u - u'\|_P^2 + \|v - v'\|_H^2
\]
and
\[\label{equ:GhatH_norm}
\| w - w'\|_{\Ghat}^2 = \|u - u'\|_{\Sigma}^2 + \|w - w'\|_G^2 = \|x - x'\|_{S + \Sigma_1}^2 + \|v - v'\|_{\hat H}^2.
\]

With the update scheme \eqref{equ:sP-PRSM2} and \eqref{equ:sP-PRSM4}, it is easy to have
\[ \label{equ:lak:lak1}
\la^k = \la^{k+1} + (\alpha + \gamma)\beta r^{k+1} + \alpha \beta B(y^k - y^{k+1}).
\]
With \eqref{equ:H} and \eqref{equ:lak:lak1}, we thus have
\begin{align}
\|v^k - v^{k+1}\|_H^2 
  ={} &  (1- \alpha) \beta  \|B(y^k - y^{k+1})\|^2 + (\alpha + \gamma) \beta \|r^{k+1}\|^2. \label{equ:lemma:vkvk1:b1}
\end{align}

\rev{Finally, it is easy to have the following proposition.
\begin{proposition}\label{proposition:H:G}
If  $0 \le \alpha \leq 1$ and $\gamma > 0$, then $H \succeq 0$. If  $T + \Sigma_2 + (1 - \alpha) \beta B^{\Tsf} B \succ 0$, then $\hat H \succ 0$. If  $T + \Sigma_2 + (1 - \alpha) \beta B^{\Tsf} B \succ 0$ and $S  + \Sigma_1 \succeq 0$, then $\Ghat \succeq 0$.
\end{proposition}
}

\section{Convergence of iPSPR }\label{section:convergence}

In this section, we first show that a sequence related to $\{w_k\}$ generated by iPSPR \eqref{equ:sP-PRSM} is strictly contractive \rev{in section \ref{subsection:contraction}} and then establish the global convergence of the method \rev{in section \ref{subsection:globalconvergence}, and discuss the choices of the proximal terms in section  \ref{subsetion:proximalterms}}.  \rev{Note that  the contraction property is also \revise{helpful} to establish} the convergence rate in the nonergodic sense.
\subsection{Contraction analysis}\label{subsection:contraction}

\rev{To establish the strictly contractive property of the sequence $\{\Phi_{\alpha, \gamma}^k (w^*) \}$ (see \eqref{equ:def:Phi} for the definition)}, we first give a rough estimation of \rev{$\|w^k - w^*\|_{\Ghat}^2 - \|w^{k+1} - w^*\|_{\Ghat}^2$} based on the optimality conditions of \eqref{equ:sP-PRSM1}  and \eqref{equ:sP-PRSM3}.

\begin{lemma} \label{lemma:vkvk1}
Let the sequence $\{w^k\}$ be generated by iPSPR \eqref{equ:sP-PRSM}. If we choose \rev{$(\alpha, \gamma) \in \dom$}, \rev{then there holds that}
\begin{align}\label{equ:lemma:vkvk1}
& \rev{\|w^k - w^*\|_{\Ghat}^2 - \|w^{k+1} - w^*\|_{\Ghat}^2}\nn \\ \geq{} &\rev{\|x^k - x^{k+1}\|_{S + \frac12 \Sigma_1}^2 + \|y^k - y^{k+1}\|_{T + \frac12 \Sigma_2}^2}
 +  (1- \alpha) \beta  \|B(y^k - y^{k+1})\|^2 \nn \\
&  + (2 - \alpha - \gamma) \beta\|r^{k+1} \|^2
  + 2(1 - \alpha)\beta \llang r^{k+1}, B(y^k - y^{k+1}) \rrang 
\end{align}
\rev{and}
 \begin{align} \label{equ:wkw:wk1:w:G}
& \rev{\|w^k - w^*\|_{\Ghat}^2 - \|w^{k+1} - w^*\|_{\Ghat}^2} \nn\\
\geq{} & \rev{\|x^k - x^{k+1}\|_{S + \frac12 \Sigma_1}^2 + \|y^k - y^{k+1}\|_{T+ \frac12 \Sigma_2}^2} \nn \\
 & \rev{+ \frac{\alpha^2(1 - \gamma) + \gamma^2 (1 - \alpha)}{(\alpha + \gamma)^2}\beta  \|B(y^k - y^{k+1})\|^2  + \frac{2 - \alpha - \gamma}{(\alpha + \gamma)^2 \beta} \left\|\lambda^k - \lambda^{k+1} \right\|^2}  \nn \\
 & +  \frac{2(\gamma - \alpha)}{(\alpha + \gamma)^2} \llang B(y^k - y^{k+1}), \lambda^k - \lambda^{k+1} \rrang.
\end{align}
\end{lemma}

\begin{proof}
\rev{The proof of \eqref{equ:lemma:vkvk1} consists of three steps.}

 \rev{I).  We give a rough lower bound estimation of the term  $\|w^{k} - w\|_{\Ghat}^2 - \|w^{k+1} - w\|_{\Ghat}^2$.  Following from the first equality of \eqref{equ:GhatH_norm}, we have
\be \label{equ:wk-w:Ghat_wk1-w:Ghat}
\|w^{k} - w\|_{\Ghat}^2 - \|w^{k+1} - w\|_{\Ghat}^2 = \|w^{k} - w\|_{G}^2 - \|w^{k+1} - w\|_{G}^2 + \|u^k - u\|_{\Sigma}^2  - \|u^{k+1} - u\|_{\Sigma}^2.
\ee
The Cauchy-Schwartz inequality ensures
$\|u^k - u\|_{\Sigma}^2  + \|u^{k+1} - u\|_{\Sigma}^2\geq \frac12 \|u^k - u^{k+1}\|_{\Sigma}^2$. Thus, we have
\[
\|u^k - u\|_{\Sigma}^2  - \|u^{k+1} - u\|_{\Sigma}^2 \geq \frac12 \|u^k - u^{k+1}\|_{\Sigma}^2 - 2\|u^{k+1} - u\|_{\Sigma}^2. \label{equ:lemma:ukuk1}
\]
 Using the identity $\|a\|_G^2 - \|b\|_G^2 = \|a - b \|_G^2 + 2b^{\Tsf}G (a - b)$  with $a = w - w^k$ and $b = w - w^{k+1}$, we have
\begin{align}
\|w^k - w\|_G^2 - \|w^{k+1} - w\|_G^2 =  \|w^k - w^{k+1} \|_G^2 + 2 (w - w^{k+1})^{\Tsf} G (w^{k+1} - w^k).  \label{equ:lemma:vkvk1:a2}
\end{align}
Substituting \eqref{equ:lemma:ukuk1} and \eqref{equ:lemma:vkvk1:a2} into  \eqref{equ:wk-w:Ghat_wk1-w:Ghat}, and using \eqref{equ:GH_norm} and \eqref{equ:lemma:vkvk1:b1},  we have that for any $w \in \Omega$}
 \begin{align}
 & \|w^k - w\|_{\Ghat}^2 - \|w^{k+1} - w\|_{\Ghat}^2\nn \\
  \geq{}& \rev{\|x^k - x^{k+1}\|_{S + \frac12 \Sigma_1}^2 + \|y^k - y^{k+1}\|_{T + \frac12 \Sigma_2}^2 +   (1- \alpha) \beta  \|B(y^k - y^{k+1})\|^2 + (\alpha + \gamma) \beta \|r^{k+1}\|^2}  \nn\\
  & \rev{+ 2 (w - w^{k+1})^{\Tsf} G (w^{k+1} - w^k) - 2\|u^{k+1} - u\|_{\Sigma}^2.} \label{equ:lemma:vkvk1:a3}
 \end{align}

\rev{II). We focus on the estimation of $(w - w^{k+1})^{\Tsf}  G (w^{k+1}-w^{k})$. 
From the optimality conditions of \eqref{equ:sP-PRSM1} and \eqref{equ:sP-PRSM3}, we know that there exist $\xi_x^{k+1} \in \partial \theta_1(x^{k+1})$ and $\xi_y^{k+1} \in \partial \theta_2(y^{k+1})$ such that}
\be \label{equ:opt:x:0}
\llang x - x^{k+1},  {\Pone}(x^{k+1} - x^k) + \xi_x^{k+1}-A^{\Tsf}\la^k + \beta  A^{\Tsf} \rev{(A x^{k +1} + B y^k - b)} \rrang \geq 0, \quad \forall x \in \X\nn
\ee
and
\[ \label{equ:opt:y:0}
\llang y - y^{k+1}, {\Ptwo}(y^{k+1} - y^k) +  \xi_y^{k+1}-B^{\Tsf}\la^{k + \frac12} + \beta  B^{\Tsf} r^{k+1} \rrang \geq 0, \quad \forall y \in \Y.
\]
\rev{Substituting \eqref{equ:lak:lak1} into \eqref{equ:opt:x:0} and noting $r^{k+1} = A x^{k+1} + B y^{k+1} - b$, we have}
\begin{align}
& \langle x - x^{k+1}, {\Pone}(x^{k+1} - x^k)  + \xi_x^{k+1} -A^{\Tsf}\la^{k+1} + (1 - \alpha - \gamma) \beta  A^{\Tsf} r^{k+1}\nn\\
& + (1 - \alpha) \beta  A^{\Tsf} B (y^k -y^{k+1} )\rangle  \geq 0, \qquad
 \forall x \in \X.  \label{equ:opt:x}
\end{align}
\rev{Substituting  $\la^{k+\frac12} = \la^{k+1} + \gamma\beta r^{k+1}$ into \eqref{equ:opt:y:0}, we have}
\[
\llang y - y^{k+1}, {\Ptwo}(y^{k+1} - y^k) +   \xi_y^{k+1} -B^{\Tsf}\la^{k + 1} + (1 - \gamma) \beta  B^{\Tsf} r^{k+1}\rrang  \geq 0, \quad \forall y \in \Y.  \label{equ:opt:y}
\]
Rewrite \eqref{equ:lak:lak1} to
\be
r^{k+1} - \frac{\alpha}{\alpha + \gamma} B(y^{k+1} - y^{k}) + \frac{1}{(\alpha + \gamma)\beta} (\la^{k+1} - \la^k) = 0. \label{equ:opt:la}
\ee
Combing \eqref{equ:opt:x}, \eqref{equ:opt:y} and \eqref{equ:opt:la} in a suitable way, and recalling the \revise{definitions} of $w$ and $F(\cdot)$,    for \rev{any $w \in \Omega$} there holds that
\begin{align}\label{equ:lemma:opt:d1}
&
 \llang w - w^{k+1},
\begin{pmatrix}
{\Pone}(x^{k+1} - x^k) \\
{\Ptwo}(y^{k+1} - y^k) \\
0
\end{pmatrix}
+
\begin{pmatrix}
0 \\
\alpha \beta B^{\Tsf}r^{k+1} + (1 - \alpha) \beta  B^{\Tsf} B (y^{k+1} - y^{k})\\
\displaystyle - \frac{\alpha}{\alpha + \gamma} B(y^{k+1} - y^{k}) + \frac{1}{(\alpha + \gamma)\beta} (\la^{k+1} - \la^k)
\end{pmatrix}
\rrang
\nn \\
\geq &
\left\langle w^{k+1} - w,
\begin{pmatrix}
  A^{\Tsf} \\
  B^{\Tsf} \\
0
\end{pmatrix}
\left[(1 - \alpha - \gamma) \beta   r^{k+1} + (1 - \alpha) \beta   B (y^k -y^{k+1} ) \right]
\rrang \nn
\\
& + \llang
w^{k+1} - w,
\rev{F(w^{k+1}, \xi_x^{k+1},  \xi_y^{k+1})}
\rrang.
\end{align}
With the assertion \eqref{equ:opt:la}, \rev{we have $\alpha \beta B^{\Tsf}r^{k+1} + (1 - \alpha) \beta  B^{\Tsf} B (y^{k+1} - y^{k}) = \frac{\alpha + \gamma - \alpha \gamma}{\alpha + \gamma}\beta B^{\Tsf} B (y^{k+1} - y^k) - \frac{\alpha}{\alpha + \gamma} B^{\Tsf} (\lambda^{k+1} - \lambda^k)$. Using} the definition \eqref{equ:H} of $H$, the definition \eqref{equ:G} of $G$ and the definition of $r^{k+1}$ and $r(w)$,
we can rewrite \eqref{equ:lemma:opt:d1} as
\begin{align}\label{equ:lemma:opt:01}
(w - w^{k+1})^{\Tsf}  G (w^{k+1}-w^{k})
\ge{} &\llang r^{k+1} -r(w), (1 - \alpha - \gamma) \beta   r^{k+1} + (1 - \alpha) \beta   B (y^k -y^{k+1} ) \rrang \nn \\ &+ \llang w^{k+1} - w, \rev{F(w^{k+1}, \xi_x^{k+1},  \xi_y^{k+1})} \rrang.
\end{align}
Noting that  $r(w) = 0$ for any $w \in \D$, we have from \eqref{equ:lemma:opt:01} that \rev{for any $w \in \D$}
\begin{align}\label{equ:lemma:opt:0}
(w - w^{k+1})^{\Tsf}  G (w^{k+1}-w^{k})
\ge{}  &(1 - \alpha - \gamma) \beta\|r^{k+1} \|^2  + (1 - \alpha)\beta \llang r^{k+1}, B(y^k - y^{k+1})\rrang \nn \\
 & + \llang w^{k+1} - w, \rev{F(w^{k+1}, \xi_x^{k+1},  \xi_y^{k+1})} \rrang.
\end{align}

\rev{III). Plugging \eqref{equ:lemma:opt:0} into \eqref{equ:lemma:vkvk1:a3}, we have \rev{for any $w \in \D$}
\begin{align}\label{equ:lemma:vkvk1:00}
& \rev{\|w^k - w\|_{\Ghat}^2 - \|w^{k+1} - w\|_{\Ghat}^2}\nn\\
 \geq{} &\rev{\|x^k - x^{k+1}\|_{S + \frac12 \Sigma_1}^2 + \|y^k - y^{k+1}\|_{T + \frac12 \Sigma_2}^2}
 +  (1- \alpha) \beta  \|B(y^k - y^{k+1})\|^2 \nn \\
&  + (2 - \alpha - \gamma) \beta\|r^{k+1} \|^2
  + 2(1 - \alpha)\beta \llang r^{k+1}, B(y^k - y^{k+1}) \rrang + \rev{\Delta(w^{k+1}, w)}, 
\end{align}
where $\rev{\Delta(w^{k+1}, w) := 2 \llang w^{k+1} - w, \rev{F(w^{k+1}, \xi_x^{k+1},  \xi_y^{k+1})} \rrang  - 2\|u^{k+1} - u\|_{\Sigma}^2.}$ Taking $w = w^{k+1}$ and $w' = w^*$ in \eqref{equ:F:mono}, we have from  \eqref{equ:F:mono}  that
\begin{align}
 \llang w^{k+1} {-} w^*, F(w^{k+1}, \xi_x^{k+1}, \xi_y^{k+1}) \rrang \geq \llang w^{k+1} {-} w^*, F(w^*, \xi_x^*, \xi_y^*) \rrang + \|u^{k+1} {-} u^* \|_{\SigT}^2 \geq \|u^{k+1} - u^*\|_{\SigT}^2, \nn
\end{align}
where $\xi_x^* \in \partial \theta_1(x^*)$, $\xi_y^* \in \partial \theta_2(y^*)$ and the second inequality is due to the optimality condition \eqref{equ:opt:cond:D} of $w^*$. \rev{This  further means that $\Delta(w^{k+1}, w^*) \geq 0$.} Setting $w = w^*$ in \eqref{equ:lemma:vkvk1:00}, we have  \eqref{equ:lemma:vkvk1}.

\rev{The proof of  \eqref{equ:wkw:wk1:w:G} follows directly from \eqref{equ:lemma:vkvk1} and $r^{k+1} = \frac{\alpha}{\alpha + \gamma} B(y^{k+1} - y^{k}) - \frac{1}{(\alpha + \gamma)\beta} (\la^{k+1} - \la^k)$ which comes from \eqref{equ:opt:la}.}
The proof is completed.}
\end{proof}

 \rev{We now need to  give a  careful estimation of the crossing term $\left\langle r^{k+1}, B(y^k - y^{k+1})\right\rangle$, which is useful to establish the strictly contractive property of $\{\Phi_{\alpha, \gamma}^k (w^*)\}$ when $(\alpha, \gamma) \in \dom_1 \cup \dom_2$  \rev{(see \eqref{equ:D1:D2} for the definition)}.}

\begin{lemma} \label{lemma:AxBy-b}
Let the sequence $\{w^k\}$ be generated by iPSPR \eqref{equ:sP-PRSM}. If $\alpha \geq 0$ and $\gamma >0$, then \rev{there holds that}
\begin{align}\label{equ:lemma:AxBy-b}
 & \left\langle r^{k+1}, B(y^k - y^{k+1})\right\rangle\nn\\
   \geq{} & \frac{1 - \gamma}{1 + \alpha}  \left\langle r^{k}, B(y^k - y^{k+1}) \right\rangle  - \frac{\alpha}{1 + \alpha} \|B(y^k - y^{k+1})\|^2  \rev{+ \frac{1}{1 + \alpha}\cdot \frac{1}{\beta}   \| y^{k} - y^{k+1} \|_{- 2T_{-} + \Sigma_2}^2}\nn\\
& + \frac{1}{2(1 + \alpha)}\cdot \frac{1}{\beta}   \left(\| y^{k} - y^{k+1} \|_{T_+ + T_{-}}^2 - \|y^{k-1} - y^{k}\|_{T_{+} + T_{-}}^2\right).
\end{align}
\end{lemma}
\begin{proof}
\rev{From the optimality conditions of  \eqref{equ:sP-PRSM3} with $k := k-1$,  we know that there exist $\xi_y^{k} \in \partial \theta_2(y^{k})$ such that}
\[\label{equ:lemma:AxBy-b:a0}
 \llang y - y^{k}, {\Ptwo}(y^k - y^{k-1}) + \xi_y^k -B^{\Tsf}\la^{k} + (1 - \gamma) \beta  B^{\Tsf} r^k \rrang \geq 0, \quad \forall y \in \Y.
\]
Setting $y$ to be $y^k$ and $y^{k+1}$ in \eqref{equ:opt:y} respectively, and \eqref{equ:lemma:AxBy-b:a0} and then rearranging the obtained inequalities suitably,  we have that
\[\label{equ:lemma:AxBy-b:a1}
\llang B(y^k - y^{k+1}),  - \la^{k + 1} + (1 - \gamma) \beta   r^{k+1} \rrang \geq \| y^k - y^{k+1} \|_{\Ptwo}^2 - \llang y^k - y^{k+1}, \xi_y^{k+1} \rrang
\]
and
\[\label{equ:lemma:AxBy-b:a2}
\llang B( y^{k} - y^{k+1}),  \la^{k} - (1 - \gamma) \beta r^k \rrang \geq - \llang y^k - y^{k+1}, {\Ptwo}(y^{k-1} - y^{k})\rrang + \llang y^{k} - y^{k+1}, \xi_y^k \rrang.
\]
Summing \eqref{equ:lemma:AxBy-b:a1} and \eqref{equ:lemma:AxBy-b:a2} over the both sides \rev{yields}
\begin{align}\label{equ:lemma:AxBy-b:d1}
& \llang B( y^{k} - y^{k+1}),  \la^{k} - \la^{k+1} \rrang +  (1 - \gamma) \beta \llang B( y^{k} - y^{k+1}),   r^{k+1} \rrang   - (1 - \gamma)\beta \llang B(y^k - y^{k+1}), r^k \rrang \nn \\
\geq{} & 
\| y^{k} - y^{k+1} \|_{\Ptwo}^2  - \llang y^{k} - y^{k+1}, {\Ptwo}(y^{k-1} - y^{k})\rrang  + \llang y^{k} - y^{k+1}, \xi_y^{k} - \xi_y^{k+1} \rrang.
\end{align}
Recalling that $T = T_{+} - T_{-}$, we know from the Cauchy-Schwarz inequality that
\begin{align}
- \llang y^{k} - y^{k+1}, T (y^{k-1} - y^{k})\rrang ={} &  - \llang y^{k} - y^{k+1}, T_{+}(y^{k-1} - y^{k})\rrang + \llang y^{k} - y^{k+1}, T_{-} (y^{k-1} - y^{k})\rrang \nn \\
\geq{} &  -\frac12 \|y^k - y^{k+1}\|_{T_{+} + T_{-}}^2   -\frac12 \|y^{k-1} - y^{k}\|_{T_{+} + T_{-}}^2,\nn
\end{align}
which with \eqref{equ:y:convex}  implies that
\begin{align}
\rev{\mbox{RHS\ of}\  \eqref{equ:lemma:AxBy-b:d1}} \geq{} &
\rev{
\frac12  \left(\| y^{k} - y^{k+1} \|_{T_{+} + T_{-}}^2 - \|y^{k-1} - y^{k}\|_{T_{+} + T_{-}}^2\right) +
\| y^{k} - y^{k+1} \|_{-2T_{-} + \Sigma_2}^2
}.\nn
\end{align}
\rev{This with  relations \eqref{equ:lak:lak1} and  \eqref{equ:lemma:AxBy-b:d1} implies that  \eqref{equ:lemma:AxBy-b}}. The proof is completed.
\end{proof}

 \rev{We now decompose the domain $\dom$ (see \eqref{stepsize} for its definition) as $\dom = \dom_1 \cup \dom_2 \cup \dom_3 \cup \dom_4$ with
\begin{align} \label{equ:D1:D2}
&\dom_1 = \left\{(\alpha, \gamma) : 0\leq \alpha < 1< \gamma <  \frac{1 - \alpha + \sqrt{(1+\alpha)^2 + 4(1 - \alpha^2)}}{2}\right\}, \nn\\ &\dom_2 = \{(\alpha, \gamma): 0 \leq \alpha < 1, \gamma = 1\}
\end{align}
and
\be
 \dom_3 = \{(\alpha, \gamma) : 0\leq \alpha < 1, 0 \leq \gamma < 1,  \alpha + \gamma > 0, \alpha \neq \gamma\},  \quad
 \dom_4 = \{(\alpha, \gamma): 0 < \alpha = \gamma < 1\}. \nn
 \ee
For  a given $w \in \D$, we define $\Phi_{\alpha, \gamma}^k (w)$  as
   \be \label{equ:def:Phi}
  \Phi_{\alpha, \gamma}^k (w) :=  \|w^k - w\|_{\Ghat}^2 + \rho_1^{\alpha, \gamma} \|y^{k-1} - y^k\|_{T_+ + T_-}^2 + \rho_2^{\alpha, \gamma} \beta \|r^k\|^2,
   \ee
 where the constants
  \be  \label{equ:constants:rho}
  \rho_1^{\alpha, \gamma} =
   \begin{cases}
 \displaystyle  \frac{1 - \alpha}{1 + \alpha} & (\alpha, \gamma) \in \dom_1,\\[2mm]
\displaystyle    \frac{1 - \alpha}{2(1 + \alpha)} &  (\alpha, \gamma) \in \dom_2,\\
    0 &  (\alpha, \gamma) \in \dom_3 \cup \dom_4,\\
      \end{cases}\quad
 \rho_2^{\alpha, \gamma} =
   \begin{cases}
\displaystyle  \frac{(\gamma - 1)^2}{(1 - c^{\alpha, \gamma})( 1 + \alpha)} &  (\alpha, \gamma) \in \dom_1,\\
  0 &  (\alpha, \gamma) \in \dom_2 \cup \dom_3 \cup \dom_4,\\
      \end{cases}
   \ee
  in which the constant  $c^{\alpha, \gamma}$ is defined as
  $$
  c^{\alpha,\gamma} \in   \begin{cases}
\displaystyle \left(0, \frac{1 - \alpha^2 + \alpha - (\alpha - 1) \gamma  - \gamma^2}{(2 - \alpha - \gamma)(1 + \alpha)}\right) &  (\alpha, \gamma) \in \dom_1,\\
  (0,1) &  (\alpha, \gamma) \in \dom_2 \cup \dom_3.\\
      \end{cases}
  $$
}

\rev{We are now ready to have the following theorem.}

\begin{theorem}\label{thm:wk:wk1:contraction}
Given $w^* \in \Omega^*$, let the sequence $\{w^k\}$ be generated by iPSPR \eqref{equ:sP-PRSM}. If we choose \rev{$(\alpha, \gamma) \in \dom$}, then  there holds that
\begin{align}\label{equ:wk:wk1:w*}
 \Phi_{\alpha, \gamma}^k (w^*) - \Phi_{\alpha, \gamma}^{k+1} (w^*)     \geq{} &    \|x^k - x^{k+1}\|_{S + \frac12 \Sigma_1}^2 + \|y^k - y^{k+1}\|_{T + \frac12 \Sigma_2 + \kappa_1^{\alpha, \gamma} (-2 T_- + \Sigma_2)+ \kappa_2^{\alpha, \gamma} \beta B^{\Tsf} B}^2  \nn \\
  &    +  \frac{ \kappa_3 ^{\alpha,\gamma}}{\beta}\|\lambda^k - \lambda^{k+1} \|^2  +  \kappa_4 ^{\alpha,\gamma} \|r^{k+1} \|^2,
  \end{align}
  \rev{where the constants
   \be \label{equ:constants:kappa}
   \kappa_1^{\alpha, \gamma} =
   \begin{cases}
\displaystyle  \frac{2(1 - \alpha)}{1 + \alpha}  & (\alpha, \gamma) \in \dom_1,\\[2mm]
 \displaystyle    \frac{1 - \alpha}{1 + \alpha} &  (\alpha, \gamma) \in \dom_2,\\
    0 & (\alpha, \gamma) \in \dom_3 \cup \dom_4,
      \end{cases}
      \quad
       \kappa_2^{\alpha, \gamma} =
   \begin{cases}
\displaystyle \frac{c^{\alpha,\gamma} (1- \alpha)^2}{1 + \alpha}  & (\alpha, \gamma) \in \dom_1,\\
 \displaystyle   \frac{c^{\alpha,\gamma}(1-\alpha)(3 - \alpha)}{4(1 + \alpha)} &  (\alpha, \gamma) \in \dom_2,\\
 \displaystyle    \frac{c^{\alpha,\gamma}(1 - \alpha)(1 - \gamma) }{(2 - \gamma - \alpha)} & (\alpha, \gamma) \in \dom_3,\\
 \displaystyle    \frac{1 - \alpha}{2} & (\alpha, \gamma) \in \dom_4,
      \end{cases}
   \ee
 and
   \be
     \kappa_3^{\alpha, \gamma} =
   \begin{cases}
  0 & (\alpha, \gamma) \in \dom_1\cup \dom_2,\\
 \displaystyle  \frac{(1 - c^{\alpha,\gamma}) (1 - \alpha) (1 - \gamma) (2 - \alpha - \gamma)}{(\gamma - \alpha)^2 + (1 - c^{\alpha,\gamma}) (1 - \alpha) (1 - \gamma) (\alpha + \gamma)^2} &  (\alpha, \gamma) \in \dom_3, \\[2mm]
 \displaystyle   \frac{1 - \alpha}{2 \alpha^2} &  (\alpha, \gamma) \in \dom_4,
      \end{cases}
        \ee
         and
  \be
  \kappa_4^{\alpha, \gamma} =
   \begin{cases}
 \displaystyle   2 - \alpha - \gamma - \frac{(\gamma - 1)^2}{(1 - c^{\alpha,\gamma})(1 + \alpha)}  & (\alpha, \gamma) \in \dom_1,\\[2mm]
\displaystyle   \frac{(1-c^{\alpha, \gamma})(1-\alpha)(3 - \alpha)}{(1 + \alpha) + (1 - c^{\alpha, \gamma})(3 - \alpha)}&  (\alpha, \gamma) \in \dom_2, \\
   0 & (\alpha, \gamma) \in \dom_3 \cup \dom_4.
      \end{cases}
  \ee
}
\end{theorem}

\begin{proof}
%
%
\rev{
%
We consider four cases.

\rev{I).~ $(\alpha, \gamma)\in \dom_1$.} By combining \eqref{equ:lemma:AxBy-b} and \eqref{equ:lemma:vkvk1}, we derive
\begin{align}
&\left(\rev{\|w^k - w^*\|_{\Ghat}^2}  + \frac{1 - \alpha}{1 + \alpha} \|y^{k-1} - y^{k} \|_{T_+ + T_{-}}^2\right)- \left(\rev{\|w^{k+1} - w^*\|_{\Ghat}^2} + \frac{1 - \alpha}{1 + \alpha} \|y^{k} - y^{k+1}\|_{T_+ + T_-}^2 \right)  \nn\\
\geq{} & \rev{\|x^k - x^{k+1}\|_{S + \frac12 \Sigma_1}^2 + \|y^k - y^{k+1}\|_{T + \frac12 \Sigma_2}^2} + \frac{2(1 - \alpha)}{1 + \alpha} \|y^k - y^{k+1}\|_{-2T_- + \Sigma_2}^2\nn \\
& +  \frac{(1- \alpha)^2}{1 + \alpha} \beta  \|B(y^{k} - y^{k+1})\|^2 + (2 - \alpha - \gamma) \beta\|r^{k+1} \|^2
   \rev{- 2(\gamma-1)}\frac{1 - \alpha}{1 + \alpha} \beta \llang r^{k}, B(y^k - y^{k+1}) \rrang.
\label{equ:vkvk1:0}
\end{align}
\rev{Note that in this case $0< c^{\alpha, \gamma}< \frac{1 - \alpha^2 + \alpha - (\alpha - 1) \gamma  - \gamma^2}{(2 - \alpha - \gamma)(1 + \alpha)} < 1$}, with the Cauchy-Schwarz inequality, we have
\[
\rev{- 2 \llang r^{k}, B(y^k - y^{k+1}) \rrang \geq  -  \frac{\gamma - 1}{(1 - \alpha)(1 - c^{\alpha, \gamma})}\cdot\|r^{k}\|^2 - \frac{(1 - \alpha)(1 - c^{\alpha, \gamma})}{\gamma - 1}\cdot\|B(y^k - y^{k+1})\|^2.} \nn
\]
Plugging the above inequality into \eqref{equ:vkvk1:0}, \rev{we obtain \eqref{equ:wk:wk1:w*} in this case}.

\rev{II). \ $(\alpha, \gamma)\in \dom_2$. For this case, \eqref{equ:lemma:AxBy-b} reduces to
\begin{align}\label{equ:lemma:AxBy-b:2}
\left\langle r^{k+1}, B(y^k - y^{k+1})\right\rangle  \geq{} &   - \frac{\alpha}{1 + \alpha} \|B(y^k - y^{k+1})\|^2  \rev{+ \frac{1}{(1 + \alpha)\beta}  \| y^{k} - y^{k+1} \|_{- 2T_{-} + \Sigma_2}^2}\nn\\
& + \frac{1}{2(1 + \alpha)\beta}    \left(\| y^{k} - y^{k+1} \|_{T_+ + T_{-}}^2 - \|y^{k-1} - y^{k}\|_{T_{+} + T_{-}}^2\right).
\end{align}
On the other hand, by the Cauchy-Schwartz inequality, we have
\be\label{equ:lemma:AxBy-b:3}
\left\langle   r^{k+1}, B(y^k - y^{k+1})\right\rangle \geq  - \delta   \|B(y^k - y^{k+1})\|^2 -  \frac{1}{4\delta} \|r^{k+1}\|^2,
\ee
where $\delta = \frac{(1 + \alpha)+ (1 - c^{\alpha, \gamma})(3 - \alpha)}{4(1 + \alpha)}$.
Combing \eqref{equ:lemma:AxBy-b:2} and \eqref{equ:lemma:AxBy-b:3}, and using \eqref{equ:lemma:vkvk1}, we have
 \begin{align} \label{equ:wk_w:new}
&\left(\rev{\|w^k - w^*\|_{\Ghat}^2}  + \frac{1 - \alpha}{2(1 + \alpha)} \|y^{k-1} - y^{k} \|_{T_+ + T_{-}}^2\right)\nn\\
& - \left(\rev{\|w^{k+1} - w^*\|_{\Ghat}^2} + \frac{1 - \alpha}{2(1 + \alpha)}  \|y^{k} - y^{k+1}\|_{T_+ + T_-}^2 \right)  \nn\\
\geq & \ \rev{\|x^k - x^{k+1}\|_{S + \frac12 \Sigma_1}^2 + \|y^k - y^{k+1}\|_{T + \frac12 \Sigma_2}^2} +  \frac{1 - \alpha}{1 + \alpha}\|y^{k} - y^{k+1}\|_{-2T_{-} + \Sigma_2}^2  \nn \\
&  +   \frac{c^{\alpha, \gamma}(1-\alpha)(3 - \alpha)}{4(1 + \alpha)}\beta  \|B(y^k - y^{k+1})\|^2 +
\revise{\frac{(1-c^{\alpha, \gamma})(1-\alpha)(3 - \alpha)}{(1 + \alpha) + (1 - c^{\alpha, \gamma})(3 - \alpha)}}\|r^{k+1}\|^2.
\end{align}
This means \eqref{equ:wk:wk1:w*} holds in this case.}

\rev{III).~ $(\alpha, \gamma) \in \dom_3$}.
Noting that $c^{\alpha, \gamma} \in (0,1)$ and letting $\tilde c = \frac{(\gamma - \alpha)^2}{(\gamma - \alpha)^2 + ( 1 - c^{\alpha, \gamma})(1 - \alpha)(1 - \gamma)(\alpha + \gamma)^2}$, we have from the Cauchy-Schwarz inequality that
\be
\rev{2(\gamma {-} \alpha) \llang B(y^k {-} y^{k+1}), \lambda^k - \lambda^{k+1} \rrang  \geq  -  \frac{(\alpha - \gamma)^2\beta}{\tilde c(2 {-} \gamma {-} \alpha)} \|B(y^k - y^{k+1})\|^2 -   \frac{\tilde c(2 {-} \gamma {-} \alpha)}{\beta} \|\lambda^k - \lambda^{k+1}\|^2}, \nn
\ee
which with \eqref{equ:wkw:wk1:w:G}  and the equality $[\alpha^2(1 - \gamma) + \gamma^2(1 - \alpha)](2 - \alpha - \gamma) = (\gamma - \alpha)^2 + (1 - \alpha)(1 - \gamma)(\alpha + \gamma)^2$ implies that   \eqref{equ:wk:wk1:w*} holds in this case.

\rev{IV).~ $(\alpha, \gamma) \in \dom_4$.  \rev{Note that $\alpha = \gamma$ in this case.} It is easy to see from \eqref{equ:wkw:wk1:w:G} that
 \begin{align}
&   \|w^k - w^*\|_{\Ghat}^2 - \|w^{k+1} - w^*\|_{\Ghat}^2 \nn \\
  \geq{} & \rev{\|x^k - x^{k+1}\|_{S + \frac12 \Sigma_1}^2 + \|y^k - y^{k+1}\|_{T + \frac12 \Sigma_2}^2} +   \frac{1-\alpha}{2}\beta  \|B(y^k - y^{k+1})\|^2  + \frac{1 - \alpha}{2 \alpha^2\beta} \left\|\lambda^k - \lambda^{k+1} \right\|^2,
\end{align}
which means that \eqref{equ:wk:wk1:w*} holds in this case. }
The proof is completed. }
\end{proof}

  \subsection{Global convergence}\label{subsection:globalconvergence}

We are now ready to state the global convergence results formally.

 \begin{theorem} \label{theorem:convergence}
    Let the sequence $\{w^k\}$ be generated by iPSPR \eqref{equ:sP-PRSM}.  If the stepsizes  \rev{$(\alpha, \gamma) \in \dom$} and the proximal terms  $S$, $T$ are chosen such that
 \be \label{equ:condition:S}
 S + \frac12 \Sigma_1 \succeq 0, \quad  \rev{S + \frac12 \Sigma_1 + \beta A^{\Tsf}A \succ 0}
 \ee
and
\be\label{equ:condition:T}
   T + \Sigma_2 + (1 - \alpha) \beta B^{\Tsf} B \succ 0,    \quad  \rev{T + \frac12 \Sigma_2 +   \kappa_1^{\alpha, \gamma} (-2 T_- + \Sigma_2) + \kappa_2^{\alpha, \gamma} \beta B^{\Tsf} B \succ 0},
    \ee
    then $\{w^{k}\}$ converges
 to an optimal solution of \eqref{prob:cp}.
  \end{theorem}
\begin{proof}
The first conditions in \eqref{equ:condition:S} and  \eqref{equ:condition:T} guarantee $\hat G \succeq 0$ and $\hat H \succ 0$, \rev{see Proposition \ref{proposition:H:G}}. We divide the proof into three steps.

I) We show that the sequences $\{w^{k}\}$ is  bounded.
It is straightforward to see from \eqref{equ:wk:wk1:w*}, \eqref{equ:condition:S} and \eqref{equ:condition:T} that \rev{$\Phi_{\alpha, \gamma}^k(w^*)$ is monotone decreasing. This with $T_+, T_{-} \succeq 0$ and the definition \eqref{equ:def:Phi} means that} $\|w^k - w^*\|_{\Ghat}^2$ is bounded. \rev{With the second equality of \eqref{equ:GhatH_norm}, we have} $\|w^k - w^*\|_{\Ghat}^2 = \|x^k - x^*\|_{S + \Sigma_1}^2 + \|v^k - v^*\|_{\hat H}^2$, \rev{which means that  $\|x^k - x^*\|_{\Pone + \Sigt}$ and $\|v^k - v^*\|_{\hat H}$ are all bounded.  Besides, with the positiveness of $\hat H$,  we know that the sequences $\{\lambda^k\}$ and $\{y^k\}$ are bounded.  Following from \eqref{equ:wk:wk1:w*}, \eqref{equ:condition:S} and \eqref{equ:condition:T}, we also have
\be \label{equ:r:limit:01}
\lim_{k \rightarrow \infty} \frac{ \kappa_3 ^{\alpha,\gamma}}{\beta}\|\lambda^k - \lambda^{k+1} \|^2      +  \kappa_4 ^{\alpha,\gamma} \|r^{k+1} \|^2  = 0.
\ee
Noting that  $\kappa_3 ^{\alpha,\gamma} + \kappa_4 ^{\alpha,\gamma} > 0$, with \eqref{equ:lak:lak1} and \eqref{equ:r:limit:01} and the boundness of $y^k$, we can see that $\{r^k\}$ is bounded.}
\rev{With the definition of $r^k$}, we know  that $\|Ax^k - Ax^*\| =  \| r^k - B(y^k - y^*) \| \leq \|r^k \| + \|B(y^k - y^*)\|$, which with the boundness of \rev{$r^k$} and  $y^k$  implies  that $\|x^k - x^*\|_{\beta A^{\Tsf}A}$ is bounded.  Recalling that \rev{$\Pone + \frac12 \Sigt + \beta A^{\Tsf}A \succ 0$} and $\|x^k - x^*\|_{\Pone + \Sigt}$ is bounded, it is safe to say that $\{x^k\}$ is also bounded.

II) We argue that any cluster point  of the sequence $\{w^k\}$ is an optimal solution of \eqref{prob:cp}. Let $\{w^{k_i}\}$ be a subsequence of the sequence $\{w^k\}$ and
$\lim\limits_{k_i \rightarrow \infty} w^{k_i}  = w^{\infty}$. \rev{Following from \eqref{equ:wk:wk1:w*}, \eqref{equ:condition:S} and \eqref{equ:condition:T}, we have
\be \label{equ:r:limit:00}
\lim_{k \rightarrow \infty} \|x^k - x^{k+1}\|_{S + \frac12 \Sigma_1} =  \lim_{k \rightarrow \infty} \|y^k - y^{k+1}\|_{T + \frac12 \Sigma_2 +   \kappa_1^{\alpha, \gamma} (-2 T_- + \Sigma_2) + \kappa_2^{\alpha, \gamma} \beta B^{\Tsf} B}
= 0.
\ee
With the second condition on $T$ in \eqref{equ:condition:T}, we know from the second equality in \eqref{equ:r:limit:00} that
\be  \label{equ:yk:limits}
\lim\limits_{k \rightarrow \infty} \|y^k - y^{k +1}\| = 0.
\ee
Again using  $\kappa_3 ^{\alpha,\gamma} + \kappa_4 ^{\alpha,\gamma} > 0$, with \eqref{equ:lak:lak1} and \eqref{equ:r:limit:01}, it is easy to see that
\be \label{equ:lambdak:rk:limits}
\lim\limits_{k \rightarrow \infty} \|r^k\| = \lim\limits_{k \rightarrow \infty} \|\lambda^k - \lambda^{k+1}\| = 0.
\ee
On the other hand, with the definition of $r^k$, we have $A(x^k - x^{k+1})  = r^k - r^{k+1} - B(y^k - y^{k+1})$. Therefore, we know from \eqref{equ:yk:limits} and \eqref{equ:lambdak:rk:limits} that  $\lim\limits_{k \rightarrow \infty} \|A (x^k - x^{k +1})\| = 0$, which with the first equality in \eqref{equ:r:limit:00} implies $\lim\limits_{k \rightarrow \infty} \|x^k - x^{k +1}\|_{S + \frac12 \Sigma_1 + \beta A^T A} = 0$. This with the second condition on $S$ in \eqref{equ:condition:S} implies
\be \label{equ:xk:limits}
\lim\limits_{k \rightarrow \infty} \|x^k - x^{k +1}\| = 0.
\ee}Since the graphs of $\partial \theta_1(\cdot)$ and $\partial \theta_2(\cdot)$ are both closed, taking the limit with respect $k_i \rightarrow \infty$ on both sides of  \rev{\eqref{equ:lemma:opt:d1}}  and by using  \eqref{equ:yk:limits}, \eqref{equ:lambdak:rk:limits} and \eqref{equ:xk:limits}, we  \rev{know that there exists $\xi_x^{\infty}$ and $\xi_y^{\infty}$ such that }
\[
  (w - w^{\infty})^{\Tsf} \rev{F(w^{\infty}, \xi_x^{\infty}, \xi_y^{\infty})} \geq 0, \quad \forall w \in \D, \nn
\]
which means that $w^{\infty}$ is an optimal solution of \eqref{prob:cp}.

III) We finally prove that the sequence $\{w^k\}$ has only one cluster point.
We first replace $w^*$ with $w^{\infty}$ in the analysis of Steps I) and II).  It follows from $\lim\limits_{k_i \rightarrow \infty} w^{k_i} = w^{\infty}$ \rev{and \eqref{equ:yk:limits}, \eqref{equ:lambdak:rk:limits}} that \rev{$\lim\limits_{k_i \rightarrow \infty}  \Phi_{\alpha, \gamma}^{k_i}(w^{\infty}) = 0$}.  Owing to the \rev{decreasing} monotonicity of the sequence $\Phi_{\alpha, \gamma}^k(w^{\infty})$, we can see that
\[
\lim_{k \rightarrow \infty} \Phi_{\alpha, \gamma}^k(w^{\infty})= 0. \nn
\]
This together with \rev{$T_{+}, T_{-} \succeq 0$ and}  $\|w^k - w^{\infty}\|_{\Ghat}^2 = \|x^k - x^{\infty}\|_{S + \Sigma_1}^2 + \|v^k - v^{\infty}\|_{\hat H}^2$ and $\hat H \succ 0$
 shows that
\[ \label{equ:lim:proof:yishengtanxi}
  \lim_{k \rightarrow \infty} \|x^k - x^{\infty}\|_{\Pone + \Sigt}  = \lim_{k \rightarrow \infty}  \|y^k - y^{\infty}\| = \lim_{k \rightarrow \infty}  \|\lambda^k - \lambda^{\infty} \|.
\]
\rev{With \eqref{equ:lemma:vkvk1:b1}, we further have $\lim_{k \rightarrow \infty} \|r^k\| =  0$}. Using again the inequality $\|Ax^k - Ax^{\infty}\| = \|r^k  - B(y^k - y^{\infty})\| \leq \|r^k \| + \|B(y^k - y^{\infty})\|$, \rev{which with \eqref{equ:lim:proof:yishengtanxi} and \eqref{equ:lambdak:rk:limits} implies}
\[\label{equ:lim:proof:yishengtanxi:2}
\lim_{k \rightarrow \infty} \|A(x^k - x^{\infty})\| = 0.
\]
Combing \eqref{equ:lim:proof:yishengtanxi} and \eqref{equ:lim:proof:yishengtanxi:2}, and using that
$\Pone + \rev{\frac12\Sigt} + \beta A^{\Tsf}A \succ 0$, we immediately have
\[
 \lim_{k \rightarrow \infty} w^k = w^{\infty}. \nn
\]
The proof is completed.\vskip -2mm
\end{proof}
\begin{remark}\rm
If \rev{the condition \eqref{equ:condition:S} is replaced by $S \succeq 0$}, we can have from $\lim\limits_{k \rightarrow \infty}\|x^k - x^{k+1}\|_{S + \frac12 \Sigma_1} = 0$ that $\lim\limits_{k \rightarrow \infty} S(x^k - x^{k+1}) = 0$. Using  similar analysis to the above proof, we can show that $\{v^k\}$ converges to some $v^{*} =  \begin{pmatrix}y^* \\ \lambda^* \end{pmatrix}$, where $w^* =  \begin{pmatrix}x^* \\ v^* \end{pmatrix}$ is an optimal solution of problem \eqref{prob:cp}.
\end{remark}
\subsection{\rev{Choices of proximal terms}} \label{subsetion:proximalterms}

\rev{When the proximal terms $S$ and $T$ satisfy conditions \eqref{equ:condition:S} and \eqref{equ:condition:T}, it is easy to see that the objective functions of subproblems \eqref{equ:sP-PRSM1} and \eqref{equ:sP-PRSM3} are strongly convex, which make the corresponding problems more easier to solve. Note that by allowing $S$ or $T$ indefinite, we can always take a larger step on updating the variable $x$ or $y$. Besides, we next show that for some special cases, with particularly chosen proximal term  $T$, the subproblem   \eqref{equ:sP-PRSM3} is easy to solve or even takes a closed form solution. Note that the discussion for the proximal term $S$ is omitted since it is similar.
}

\rev{
We consider to choose $T$ as
\[\label{equ:T:new}
T = r I - \left(\Sigma_2 + \beta B^{\Tsf} B\right)\quad \mbox{with}\quad  r = \lambda_{\max}\left(\frac12 \Sigma_2 + \tau \beta B^{\Tsf} B\right),
\]
where $\tau \in (0,1]$.  We decompose $T = T_+ - T_{-}$  with
$T_+ = r I - (\frac12 \Sigma_2 + \tau \beta B^{\Tsf} B)$ and $T_- = \frac12 \Sigma_2 + (1 - \tau) \beta B^{\Tsf} B$. Note that $T_+, T_{-} \succeq 0$.  By some direct calculations,  we have
\be \label{equ:T:choice:1}
  T + \Sigma_2 + (1 - \alpha) \beta B^{\Tsf} B = r I  - \alpha \beta B^{\Tsf}B
\ee
and
\be \label{equ:T:choice:2}
T + \frac12 \Sigma_2 +   \kappa_1^{\alpha, \gamma} (-2 T_- + \Sigma_2) + \kappa_2^{\alpha, \gamma} \beta B^{\Tsf} B = rI - \left(\frac12 \Sigma_2 + (1 + 2 \kappa_1^{\alpha, \gamma}(1 - \tau) - \kappa_2^{\alpha,\gamma})\beta B^{\Tsf}B\right).
\ee
For given $(\alpha, \gamma) \in \dom$ and a fixed $c^{\alpha, \gamma}$, by \eqref{equ:T:choice:1}  and \eqref{equ:T:choice:2}, we know that if we choose
$\tau > \alpha$ and $\tau > 1 - \frac{\kappa_2^{\alpha,\gamma}}{1 + 2 \kappa_1^{\alpha, \gamma}}$,  then \eqref{equ:condition:T}  must hold.  Note that  the number $1 - \frac{\kappa_2^{\alpha,\gamma}}{1 + 2 \kappa_1^{\alpha, \gamma}}$ is decreasing with respect to $c^{\alpha, \gamma}$ which is defined over an open interval. Hence, we can argue that if
$$
1 \geq \tau >   \max \left\{\alpha, \inf_{c^{\alpha, \gamma}} \left\{1 - \frac{\kappa_2^{\alpha,\gamma}}{1 + 2 \kappa_1^{\alpha, \gamma}} \right\} \right\},
$$
namely,
\be  \label{equ:tau}
1 \geq \tau > \underline{\tau}^{\alpha,\gamma} :=
\begin{cases}\displaystyle
1 - (1-\alpha)^2\frac{1-\alpha^2 - (\gamma-1)(\alpha+\gamma)}{(2-\alpha-\gamma)(1+\alpha)(5-3\alpha)} & (\alpha, \gamma) \in \dom_1, \\
\displaystyle \frac{3 + \alpha}{4} & (\alpha, \gamma) \in \dom_2, \\[2mm]
\displaystyle \frac{1 - \alpha \gamma}{2 - \alpha - \gamma} & (\alpha, \gamma) \in \dom_3, \\
\displaystyle \frac{1 + \alpha}{2} & (\alpha, \gamma) \in \dom_4,
\end{cases}
\ee
  then \eqref{equ:condition:T}  must hold.}

\rev{Consider the case  when $\theta_2(y) = \frac12 y^{\Tsf} M y + h(y)$,  where $M$ is symmetric positive semidefinite  and the nonsmooth convex function $h(y)$ is simple in the sense that $\min_{y \in \Y} h(y) + \frac{1}{2}\|y - d\|^2$ is easy to compute. Here $d \in \R^{n_2}$ is a given vector. In this case, we have $\Sigma_2 = M$ and the subproblem \eqref{equ:sP-PRSM2} with $T$ chosen according to \eqref{equ:T:new} and \eqref{equ:tau} takes the following form
\be
y^{k+1} = \arg\min_{y\in \Y}\ h(y) + \frac12 \|y - d^k\|^2\nn
\ee
with $d^k = Ty^k + B^{\Tsf} \left(\lambda^{k + \frac12} - \beta (Ax^{k+1} - b)\right)$.
}

\rev{To end this \revise{subsection}, some \revise{comments} are listed in order. \revise{Firstly}, if $\alpha = 0, \gamma = 1$ and $\Sigma_2 = 0$, \eqref{equ:tau} becomes $0.75 < \tau \leq 1$, which recovers the optimal bound of $\tau$ for the linearized ADMM in \cite{he2019optimally}; if $\alpha \in (0,1), \gamma = 1$ and $\Sigma_2 = 0$, \eqref{equ:tau} becomes $(3 + \alpha)/4 < \tau \leq 1$, which partially recovers the optimal bound of $\tau$ for the linearized version of the generalized ADMM in \cite{jiang2018generalized}. Note that in \cite{jiang2018generalized}, they \revise{allowed} $\alpha \in (-1,1)$. \revise{Secondly}, if $(\alpha, \gamma) \in \dom_2 \cup \dom_3$, it is easy to see that $\frac{1 - \alpha \gamma}{2 - \alpha - \gamma} \geq \frac{2 + \alpha + \gamma}{4}$ and the equality holds if and only if $\alpha = \gamma$, namely, $(\alpha, \gamma) \in \dom_4$. \revise{Thirdly}, when the subproblem \eqref{equ:sP-PRSM3} does not take a closed form solution, as done in \cite{li2016majorized,fazel2013hankel, zhang2017linearly}, we can consider the majorized version of iPSPR. The techniques for constructing the indefinite proximal $T$ in \cite{li2016majorized, chen2021equivalence} can be explored to construct $T$. We leave this for future investigation.
}

\section{Sublinear Convergence of iPSPR}\label{section:sublinear}
The rate of convergence of an algorithm can help us have a deeper understanding of the algorithm. \rev{In this section,  motivated by \cite{li2016majorized, deng2017parallel,han2017linear,zhang2017linearly,chen2021equivalence}, we establish the $o(1/t)$ sublinear rate of convergence of iPSPR in the  nonergodic sense}.

We first give a new optimality condition of \eqref{prob:cp} as follows.
  \begin{lemma} \label{lemma:opt}
    Let the sequence $\{w^k\}$ be generated by iPSPR \eqref{equ:sP-PRSM}.   We choose $(\alpha, \gamma) \in \dom$ and the proximal terms  $S$, $T$ are chosen such that  \eqref{equ:condition:S} and \eqref{equ:condition:T} hold. \rev{Then $w^{k+1} \in \Omega^*$}, namely, $w^{k+1}$ is one optimal solution of \eqref{prob:cp},  if $$\|w^k - w^{k+1}\|_{\Ghat} = 0.$$
  \end{lemma}
  \begin{proof}
\rev{The proof is similar to the second part of the proof of Theorem \ref{theorem:convergence}, we omit the details here.}
  \end{proof}

  \rev{Following from \eqref{equ:H:hat}, \eqref{equ:GhatH_norm}  and \eqref{equ:lemma:vkvk1:b1}, we have
  $$
  \|w^k - w^{k+1}\|_{\Ghat}^2 = \|x^k - x^{k+1}\|_{\Pone + \Sigma_1}^2 + \|y^k - y^{k+1}\|_{T + \Sigma_2 + (1 - \alpha)\beta B^TB}^2 + (\alpha + \gamma) \beta \|r^{k+1}\|^2.
  $$
  \rev{Hence, Lemma \ref{lemma:opt} provides a practical stopping condition for iPSPR \eqref{equ:sP-PRSM}, which is shown as}
  \[
   \rev{\max\{\|x^k - x^{k+1}\|_{\Pone + \Sigma_1},\|y^{k} - y^{k+1} \|_{T + \Sigma_2 + (1 - \alpha)\beta B^TB}^2,  \beta \|r^{k+1}\|^2\} \leq \mathrm{tol}},
  \]
  where $\mathrm{tol}$ is some tolerance.
  }

\begin{theorem}\label{theorem:nonergodic:1}
\rev{Let the sequence $\{w^k\}$ be generated by iPSPR \eqref{equ:sP-PRSM} \rev{with $(\alpha, \gamma) \in \dom$}.  Suppose that
the proximal terms  $S$, $T$ are chosen such that  \eqref{equ:condition:S}, \eqref{equ:condition:T}
and
\be \label{equ:S:new}
S + \frac12\Sigma_1 \succeq \frac12 c \Sigma_1
\ee
hold, where $c$ is a positive constant.
We have  that}
   \[
    \rev{\min_{1 \leq i \leq t} \|w^{i} - w^{i+1} \|_{\Ghat}^2 = o(1/t).} \label{theorem:nonergodic:00:min}
    \]
\end{theorem}
\begin{proof}
\rev{With conditions \eqref{equ:S:new} on $S$, we know that $S + \Sigma_1 \preceq  (1 + c^{-1}) (S + \frac12 \Sigma_1)$. With the condition \eqref{equ:condition:T} on $T$, we know that there exists a positive constant $c_1$ such that
$$\hat H \preceq c_1 \begin{pmatrix}
T + \frac12 \Sigma_2 + \kappa_1^{\alpha, \gamma} (-2 T_- + \Sigma_2)+ \kappa_2^{\alpha, \gamma} \beta B^{\Tsf} B & 0 \\
 0 & \frac{\bar \kappa_3 ^{\alpha,\gamma}}{\beta} I
\end{pmatrix},
$$
which with \eqref{equ:Ghat} implies that
\[\label{equ:condition:Ghat:nonergodic}
\Ghat \preceq \max\{1 + c^{-1}, c_1\}
\begin{pmatrix}
S + \frac12 \Sigma_1& 0 & 0 \\
0 & T + \frac12 \Sigma_2 + \kappa_1^{\alpha, \gamma} (-2 T_- + \Sigma_2)+ \kappa_2^{\alpha, \gamma} \beta B^{\Tsf} B & 0 \\
0 & 0 & \frac{\bar \kappa_3 ^{\alpha,\gamma}}{\beta} I
\end{pmatrix}. \nn
\]
This with \eqref{equ:wk:wk1:w*} implies that}
\begin{align}\label{equ:wk:wk1:w*:22}
 \Phi_{\alpha, \gamma}^k (w^*) - \Phi_{\alpha, \gamma}^{k+1} (w^*)
  \geq{}& \frac{1}{\max\{c, c_1\}} \|w^k - w^{k+1}\|_{\Ghat}^2.
  \end{align}
  Summing \eqref{equ:wk:wk1:w*:22} over $k = 1, \ldots, +\infty$ leads to
  \[\label{theorem:nonergodic:b0}
    \frac{1}{\rev{c}} \cdot   \sum_{k=1}^{+\infty} \|w^{k+1} - w^{k} \|_{\Ghat}^2 \leq \Phi_{\alpha, \gamma}^1 (w^*).
    \]
\rev{Using Lemma 3 in \cite{li2016majorized}, we have \eqref{theorem:nonergodic:00:min}.}
\end{proof}

\rev{Now we show that if $(\alpha, \gamma) \in \dom_2 \cup \dom_3 \cup \dom_4$ and some additional requirement is made on $T$, we can have a stronger result than \eqref{theorem:nonergodic:00:min}}. We first show that the sequence $\{\|w^k- w^{k+1}\|_{\Ghat}^2\}$ is non-increasing.
\begin{lemma}\label{lemma:decrease}
Let the sequence $\{w^k\}$ be generated by iPSPR \eqref{equ:sP-PRSM}. If \rev{$(\alpha, \gamma) \in \dom_2 \cup \dom_3 \cup \dom_4$} and
the proximal terms  $S$, $T$ are chosen such that  \eqref{equ:condition:S}, \eqref{equ:condition:T} and
\be\label{equ:T:nonergodic}
T+ \frac12 \Sigma_2 + \frac{(1 - \alpha)(1 - \gamma)}{2 - \alpha - \gamma} \beta B^{\Tsf}B \succeq 0
\ee
hold,
 then there holds that
\[ \label{equ:lemma:decrease:00}
\|w^k - w^{k+1} \|_{\Ghat}^2 \geq \|w^{k+1} - w^{k+2} \|_{\Ghat}^2.
    \]
\end{lemma}
\begin{proof}
Note that \eqref{equ:lemma:opt:01} also holds with $k := k +1$, then we have
\rev{
\begin{align}\label{equ:lemma:opt:nonergodic:2}
& (w - w^{k+2})^{\Tsf}  G (w^{k+2}-w^{k+1})\nn\\[2mm]
\ge{} & \llang r^{k+2} -r(w), (1 - \alpha - \gamma) \beta   r^{k+2} + (1 - \alpha) \beta   B (y^{k+1} -y^{k+2}) \rrang  \nn \\[2mm]
&+  \llang w^{k+2} - w, \rev{F(w^{k+2}, \xi_x^{k+2},  \xi_y^{k+2})} \rrang,
\end{align}
where $\xi_x^{k+2} \in \partial \theta_1(x^{k+2})$ and $\xi_y^{k+2} \in \partial \theta_2(y^{k+2})$.
Choosing $w$ to be $w^{k+2}$ and $w^{k+1}$, respectively, in \eqref{equ:lemma:opt:01} and \eqref{equ:lemma:opt:nonergodic:2} leads to
\begin{align}
 & (w^{k+2} - w^{k+1})^{\Tsf}  G (w^{k+1}-w^{k})\nn\\[2mm]
\ge{}& \llang r^{k+1} -r^{k+2}, (1 - \alpha - \gamma) \beta   r^{k+1} + (1 - \alpha) \beta   B (y^{k} -y^{k+1}) \rrang \nn \\[2mm]
 & + \llang w^{k+1} - w^{k+2}, \rev{F(w^{k+1}, \xi_x^{k+1},  \xi_y^{k+1})} \rrang. \label{equ:lemma:opt:nonergodic:b0}
\end{align}
and
\begin{align}
& (w^{k+1} - w^{k+2})^{\Tsf}  G (w^{k+2}-w^{k+1})\nn\\[2mm]
\ge{} &\llang r^{k+2} -r^{k+1}, (1 - \alpha - \gamma) \beta   r^{k+2} + (1 - \alpha) \beta   B (y^{k+1} -y^{k+2}) \rrang\nn \\[2mm]
&+ \llang w^{k+2} - w^{k+1}, F(w^{k+2}, \xi_x^{k+2},  \xi_y^{k+2}) \rrang. \label{equ:lemma:opt:nonergodic:b1}
\end{align}

Adding \eqref{equ:lemma:opt:nonergodic:b0} and \eqref{equ:lemma:opt:nonergodic:b1} and noting
$\langle w^{k+2} - w^{k+1}, F(w^{k+2}, \xi_x^{k+2},  \xi_y^{k+2}) - F(w^{k+1}, \xi_x^{k+1},$ $  \xi_y^{k+1}) \rangle \geq \|u^{k+2} - u^{k+1}\|_{\Sigma}^2$ which follows from \eqref{equ:F:mono},   we obtain that
\begin{align}
&(w^{k+2} - w^{k+1})^{\Tsf}  G \left[(w^{k+1}-w^{k}) - (w^{k+2} - w^{k+1}) \right] \nn \\[2mm]
\geq{}& (1 - \alpha - \gamma) \beta \| r^{k+1} - r^{k+2}\|^2  + (1 - \alpha)\beta \left \langle B\left[(y^k - y^{k+1}) - (y^{k+1} - y^{k+2})\right], r^{k+1} - r^{k+2}\right\rangle. \nn\\[2mm]
  & + \|u^{k+2} - u^{k+1}\|_{\Sigma}^2.
   \label{equ:lemma:opt:nonergodic:b3}
\end{align}}
Following the deriving process of \rev{\eqref{equ:GH_norm} and} \eqref{equ:lemma:vkvk1:b1}, we have
\begin{align}
&  \|(w^{k}-w^{k+1}) - (w^{k+1} - w^{k+2} )\|_G^2    \nn \\
={} &  \|(u^k - u^{k+1}) - (u^{k+1} - u^{k+2})\|_{P}^2 + (1-\alpha)\beta \left\|B[ (y^{k} - y^{k+1}) - (y^{k+1} - y^{k+2})]\right\|^2 \nn \\
{} & + (\alpha + \gamma) \beta \|r^{k+1} - r^{k+2}\|^2.\label{equ:lemma:opt:nonergodic:b4}
\end{align}
Thus we conclude that
\begin{align}
 & \rev{\|w^{k} - w^{k+1} \|_{\Ghat}^2 - \|w^{k+1} - w^{k+2} \|_{\Ghat}^2} \nn \\
  ={}& (\|w^{k} - w^{k+1} \|_{G}^2 + \|u^k - u^{k+1}\|_{\Sigma}^2) - (\|w^{k+1} - w^{k+2} \|_{G}^2  + \|u^{k+1} - u^{k+2}\|_{\Sigma}^2)\nn \\
  ={}& 2(w^{k+2} - w^{k+1})^{\Tsf}  G \left[(w^{k+1}-w^{k}) - (w^{k+2} - w^{k+1}) \right] + \|(w^{k+1}-w^{k}) - (w^{k+2} - w^{k+1} )\|_{G}^2 \nn \\
  & +
    \|u^k - u^{k+1}\|_{\Sigma}^2 - \|u^{k+1} - u^{k+2}\|_{\Sigma}^2 \nn \\
  \ge{}& (2 - \alpha - \gamma) \beta \| r^{k+1} - r^{k+2}\|^2  + 2(1 - \alpha)\beta \left \langle B\left[(y^k - y^{k+1}) - \rev{(y^{k+1} - y^{k+2}})\right], r^{k+1} - r^{k+2}\right\rangle \nn \\
  & + (1-\alpha)\beta \left\|\rev{B[ (y^{k} - y^{k+1}) - (y^{k+1} - y^{k+2})]}\right\|^2 + \|(u^k - u^{k+1}) - (u^{k+1} - u^{k+2})\|_{P}^2 \nn \\
  & +  \|u^k - u^{k+1}\|_{\Sigma}^2+  \|u^{k+1} - u^{k+2}\|_{\Sigma}^2\nn \\
  \geq{} &\frac{(1 - \alpha)(1 - \gamma)}{2 - \alpha - \gamma} \beta \left\| \rev{B[(y^{k} - y^{k+1}) - (y^{k+1} - y^{k+2})]}\right\|^2 + \|(u^k - u^{k+1}) - (u^{k+1} - u^{k+2})\|_{P}^2\nn  \\
  & +
   \frac12 \|(u^k - u^{k+1}) - (u^{k+1} - u^{k+2})\|_{\Sigma}^2 \nn \\
={}&  \|(x^k - x^{k+1}) - (x^{k+1} - x^{k+2})\|_{S + \frac12 \Sigma_1}^2 +  \|(y^k - y^{k+1}) - (y^{k+1} {-} y^{k+2})\|_{T+ \frac12 \Sigma_2 + \frac{(1 {-} \alpha)(1 {-} \gamma)}{2 - \alpha - \gamma} \beta B^{\Tsf}B}^2\nn \\
\geq{}& 0.
 \end{align}
where the first inequality is due to \eqref{equ:lemma:opt:nonergodic:b3} and \eqref{equ:lemma:opt:nonergodic:b4}, the second inequality follows from the Cauchy-Schwarz inequality and the last inequality is due to $P = \left(\begin{matrix}  S& 0 \\ 0 & T\end{matrix} \right)$, $S + \frac12 \Sigma_1 \succeq 0$ and \eqref{equ:T:nonergodic}.
The proof is completed.
\end{proof}

\begin{theorem}\label{theorem:nonergodic:2}
Let the sequence $\{w^k\}$ be generated by iPSPR \eqref{equ:sP-PRSM} with $(\alpha, \gamma) \in \dom_2 \cup \dom_3 \cup \dom_4$.  Suppose that
the proximal term  $S$ is chosen  according to \eqref{equ:condition:S} and \eqref{equ:S:new} and the proximal term $T$ is chosen according to \eqref{equ:condition:T} and \eqref{equ:T:nonergodic}.
We have
  \[
    \|w^{t} - w^{t+1} \|_{\widehat G}^2 = o(1/t). \label{theorem:nonergodic:00}
    \]
\end{theorem}
\begin{proof}
It follows from Theorem \ref{theorem:nonergodic:1} that $\min_{1 \leq i \leq t}\|w^{i} - w^{i+1} \|_{\widehat G}^2 = o(1/t)$, which with \eqref{equ:lemma:decrease:00} implies \eqref{theorem:nonergodic:00}.
The proof is completed.
\end{proof}

\revi{
Define the KKT mapping $R\colon \Omega \rightarrow \Omega$ as
\[
R(w) = \left(\begin{array}{c}
x - \mathrm{Pr}_{p}[x -(\nabla f(x) -A^{\Tsf} \lambda)] \\ y - \mathrm{Pr}_{h}[y - (\nabla g(y) -B^{\Tsf} \lambda)] \\ Ax + By - b \end{array} \right), \;\; \forall w \in \Omega,
\]
where $\mathrm{Pr}_{p}(\cdot)$ denotes the Moreau-Yosida proximal mapping \cite{RW2004} defined as
$$\mathrm{Pr}_{p}(z):= \arg\min_{x \in \X} \left\{p(x) + \frac12 \|x - z\|^2 \right\},
$$
and  $\mathrm{Pr}_{h}(\cdot)$ is defined accordingly.  It is well known that (see \cite{han2017linear,zhang2017linearly,chen2021equivalence} for instance)
\[\forall w \in \Omega, \quad R(w) = 0 \Longleftrightarrow w \in \Omega^*.
\]
Inspired by \cite[Lemma 1]{han2017linear} and \cite[Lemma 3.1]{zhang2017linearly}, we now characterize the relations between $\|R(w^{k+1})\|^2$ and $\|w^{k+1} - w^k\|_{\widehat G}^2$.
\begin{lemma}\label{lemma:kkt}
Let the sequence $\{w^k\}$ be generated by iPSPR \eqref{equ:sP-PRSM}. Then there exists  a positive constant $\varrho$ such that
for any $k\geq 0$,
\be\label{lemma:kkt:main}
\|R(w^{k+1})\|^2  \leq \varrho \|w^{k+1} - w^k\|_{\widehat G}^2.
\ee
\end{lemma}
\begin{proof}
From the optimality condition for \eqref{equ:sP-PRSM} and $\lambda^{k + \frac12} = \lambda^{k+1} + \gamma \beta r^{k+1}$, we have
\[\label{defx}
x^{k+1} =  \mathrm{Pr}_{p}\left[x^{k+1} - \left(\nabla f(x^{k+1})-A^{\Tsf}\la^k + \beta A^{\Tsf} r^{k+1} + \beta  A^{\Tsf}  B (y^k - y^{k+1})+ S(x^{k+1} - x^k) \right)\right]
\]
and
\[\label{defy}
y^{k+1} =  \mathrm{Pr}_{h}\left[y^{k+1} - \left(\nabla g(y^{k+1})-B^{\Tsf}\la^{k+1} + (1 - \gamma) \beta  B^{\Tsf} r^{k+1} + T(y^{k+1} - y^k) \right)\right].
\]
Substituting  the above equations \eqref{defx} and \eqref{defy} into $R(w^{k+1})$ and
noting that  the Moreau-Yosida proximal mapping is globally  Lipschitz continuous with modulus one, we get
\begin{align}
\|R(w^{k+1})\| \leq{}&  \left\|
A^{\Tsf} (\lambda^{k+1} - \lambda^{k})  + \beta A^{\Tsf} r^{k+1} + \beta  A^{\Tsf}  B (y^k - y^{k+1}) + S(x^{k+1} - x^k) \right\| \nn \\
&    + \left\| (1 - \gamma) \beta  B^{\Tsf} r^{k+1}  + T(y^{k+1} - y^k) \right\| + \| r^{k+1}\| \nn \\
\leq{} & \|A\|  \|\lambda^{k+1} - \lambda^{k}\| + \left( (\|A\| + |1 - \gamma| \|B\|)\beta + 1\right) \|r^{k+1}\| + \beta \|A^{\Tsf}B\| \|y^k - y^{k+1}\| \nn \\
& + \|S(x^{k+1} - x^k)\| + \|T(y^k - y^{k+1})\|.  \nn
\end{align}
Notice that $\|S(x^{k+1} - x^k)\|^2 \leq \|S\| \|x^{k+1} - x^k\|_{S + \Sigma_1}^2$, by using \eqref{equ:opt:la},  we thus have
\be\label{lemma:kkt:1}
\|R(w^{k+1})\| \leq  \iota_1 \|x^{k+1} - x^k\|_{S + \Sigma_1}  + \iota_2 \|y^{k+1} - y^k\| + \iota_3 \|\lambda^{k+1} - \lambda^k\|,
\ee
where $\iota_1 = \sqrt{\|S\|}$, $\iota_2 =  \left((\|A\| + |1 - \gamma| \|B\|)\beta + 1\right)  \frac{\alpha \|B\|}{\alpha + \gamma}  + \beta \|A^{\Tsf}B\| + \|T\|$
and $\iota_3 = \|A\| + \frac{(\|A\| + |1 - \gamma| \|B\|)\beta + 1}{(\alpha + \gamma)\beta} $.
On the other hand, we know from \eqref{equ:GhatH_norm} that
\be\label{lemma:kkt:2}
\|w^{k+1} - w^k\|_{\widehat G}=  \|x^{k+1} - x^k\|_{S + \Sigma_1}^2 +\|v^{k+1} - v^k\|_{\widehat H}^2.
\ee
Since $\widehat H\succ 0$ (see Proposition \eqref{proposition:H:G}),  combining \eqref{lemma:kkt:1} and \eqref{lemma:kkt:2} together,  we know that there must exist a positive constant such that  \eqref{lemma:kkt:main} holds. The proof is completed.
\end{proof}
With Theorem \ref{theorem:nonergodic:1}, Theorem \ref{theorem:nonergodic:2}  and Lemma \ref{lemma:kkt}, we can immediately arrive at the sublinear convergence rate results  based on the KKT residual $\|R(w^k)\|$.
\begin{theorem}\label{theorem:nonergodic:2n}
Let the sequence $\{w^k\}$ be generated by iPSPR \eqref{equ:sP-PRSM} {with $(\alpha, \gamma) \in \dom$}.  Suppose that
the proximal term  $S$ is chosen  according to \eqref{equ:condition:S} and \eqref{equ:S:new} and the proximal term $T$ is chosen according to \eqref{equ:condition:T} hold, we have
   $$
   \min_{1 \leq i \leq t} \|R(w^{i+1})\|^2 = o(1/t). 
    $$
If we restrict $(\alpha, \gamma) \in \dom_2 \cup \dom_3 \cup \dom_4$ and suppose that
the proximal term  $S$ is chosen  according to \eqref{equ:condition:S} and \eqref{equ:S:new} and the proximal term $T$ is chosen according to \eqref{equ:condition:T} and \eqref{equ:T:nonergodic}, we have
  $$
    \|R(w^{t+1}) \|^2 = o(1/t). 
    $$
\end{theorem}
}

 \section{Numerical Results}\label{section:numerical}
In this section, we demonstrate the  potential efficiency of our method iPSPR \eqref{equ:sP-PRSM} by \rev{solving the} following $\ell_1$ regularized least square problem
\[\label{equ:prob:qp:0}
\min\  \frac12 \|Q y - c\|^2 + \rho \|y\|_1, \quad \st \quad  By \leq b,
\]
where $y \in \R^n, c \in \R^p$, $Q \in \R^{p \times n}$ and $B \in \R^{m \times n}$.
Problem \eqref{equ:prob:qp:0} is a constrained extension of the ordinary unconstrained $\ell_1$ regularized least square problem and it was considered in \cite{li2016majorized}.  By introducing an auxiliary variable $x \in \R^m$, we rewrite \eqref{equ:prob:qp:0} as
\[\label{equ:prob:qp}
\min\  \frac12 \|Q y - c\|^2 + \rho \|y\|_1, \quad \st \quad  x + By = b,  \ x \geq 0,
\]
which is a special instance  of \eqref{prob:cp}.

\rev{For our method iPSPR \eqref{equ:sP-PRSM}, we set $S = 0$ and choose $T$ according to \eqref{equ:T:new}, namely,
\be\label{equ:T:num}
T = rI - (Q^{\Tsf}Q + \beta B^{\Tsf} B) \quad \mbox{with}\quad r = \lambda_{\max}\left(\frac12 Q^{\Tsf} Q + \tau \beta B^{\Tsf}B
\right)
\ee
with $\tau = 1.001\underline{\tau}^{\alpha, \gamma}$, where  $\underline{\tau}^{\alpha, \gamma}$ is defined in \eqref{equ:tau}.} Our method iPSPR \eqref{equ:sP-PRSM} for solving \eqref{equ:prob:qp} is then given as
\[ \label{equ:prob:qp:iPSPR}
\left\{
\begin{aligned}
x^{k+1}&= \mathcal{P}_{+}\left[b - B y^k  + \lambda^k/\beta\right],\\
\la^{k+\frac{1}{2}}&=\la^k-\alpha\beta(x^{k+1} + B y^{k} - b),\\
y^{k+1}&= \rev{\mathcal{S}_{\rho/r}\left[y^k + \frac{1}{r} \left(B^{\Tsf}\left(\lambda^{k + \frac12} - \beta (x^{k+1} + B y^k - b) \right) + Q^{\Tsf}(c -  Q y^k)\right)\right]},\\
\la^{k+1}&=\la^{k+\frac{1}{2}}-\gamma\beta(x^{k+1} + B y^{k+1} - b),
\end{aligned}
\right.\]
where \rev{the projection operator $\mathcal{P}_{+}(z) = \max(z, 0)$ and} the shrinkage operator  $\mathcal{S}_{\nu} (z) := \mathrm{sgn}(z) \odot \max\{|z| - \nu, 0\}$. \rev{Note that for problem \eqref{equ:prob:qp}, the majorized indefinite proximal ADMM in \cite{li2016majorized} coincides with our iPSPR \eqref{equ:prob:qp:iPSPR}  with $\alpha = 0$ and $(\alpha, \gamma) \in \dom_1$ since the smooth part of the objective function is quadratic with respect to $y$.}  \rev{If the proximal parameter $r = 1.001\lambda_{\max}(Q^{\Tsf} Q + \beta B^{\Tsf} B)$, iPSPR becomes the semidefinite  proximal-based strictly contractive Peaceman-Rachford splitting method (sPSPR).  \rev{To make a fair comparison, as done in (85) of \cite{li2016majorized}, we stop iPSPR or sPSPR when the KKT residual is less than $10^{-6}$.}
}

All the experiments were preformed in Ubuntu 16.04 LTS  a Dell workstation with a 3.5-GHz Intel Xeon E3-1240 v5 processor with access to 32 GB of RAM.  All the  methods
were implemented in MATLAB (R2016b). Given $m$ and $n$, as done in \cite{li2016majorized}, we set $p = 0.1n$, $\rho = 5 \sqrt{n}$ and generate the data as
\begin{verbatim}
  B = sprandn(m, n, 0.2);  yy = randn(n, 1);  b = B * yy + max(randn(m,1), 0);
  Q = sprandn(p, n, 0.1);  c = Q * yy;
\end{verbatim}

In our tests, we \rev{set $m = 2000$ and $n = 1000, 2000, 4000, 8000$. For each $m$ and $n$, we} use the above scheme to generate 50 groups of instances and will always report the average performance for methods  iPSPR and sPSPR. \rev{\revise{For each instance, we fix the sum $\alpha + \gamma$ to be $\{1.9, 1.8, 1.618, 1\}$ and always choose the special cases with $\alpha = \gamma$,} \rev{$\alpha = 0$ or $\gamma = 1$.} \rev{In total, we have} \revise{nine groups of choices of $\alpha$ and $\gamma$.}
 In our tests,  the  penalty parameter $\beta$ is fixed during the iterations. Generally, choosing the best penalty parameter $\beta$ is not easy and it might be problem dependent \cite{he2014strictly}. We spent some efforts to choose the penalty parameter $\beta$ from a large number of candidates. For each given $m, n$ and $\alpha, \gamma$, we report the performance of iPSPR or sPSPR with four choices of $\beta$. Note that in our tests, the second choice is the best choice in the candidates for $\alpha + \gamma = 1.9$ and \revise{almost} the best choice in the candidates for $\alpha + \gamma \in \{1.618, 1.8\}$; the third choice of $\beta$ is the best choice in the candidates for $\alpha + \gamma = 1$.}
\vskip 3mm

\renewcommand{\arraystretch}{1.1}
\begin{table}[!htb]
	\setlength{\tabcolsep}{2pt}
\centering\small
 \caption{\rev{The results for $m = 2000, n = 1000$ over 50 runs. The CPU time  is in seconds.}}
\rev{
\begin{tabular}{|cc|rrr|rrr|rrr|rrr|}
 %
 \hline
  &&   \multicolumn{3}{c|}{$\beta =  0.50$}    & \multicolumn{3}{c|}{$\beta = 1.50$} & \multicolumn{3}{c|}{$\beta = 3.00$} &    \multicolumn{3}{c|}{\revise{$\beta = 5.00$}} \\
 \cline{3-14}
 $(\alpha, \gamma)$ & method & iter & $r$ & t  & iter & $r$ & t  & iter & $r$ & t  & iter & $r$ & t  \\   \hline
(0.950, 0.950) & iPSPR & 8769.0 & 5.23e2 &  4.9 & \textbf{4750.1} & 1.56e3 &  2.7 & 5876.0 & 3.11e3 &  3.3 & 8576.2 & 5.18e3 &  4.8 \\
(0.950, 0.950) & sPSPR & 8877.3 & 5.47e2 &  5.0 & \textbf{4815.4} & 1.60e3 &  2.7 & 6000.3 & 3.19e3 &  3.4 & 8810.7 & 5.31e3 &  5.0 \\  \cdashline{1-14}[0.8pt/2pt]
(0.900, 1.000) & iPSPR & 8769.8 & 5.23e2 &  5.0 & \textbf{4752.9} & 1.56e3 &  2.7 & 5878.7 & 3.11e3 &  3.4 & 8579.8 & 5.18e3 &  4.9 \\
(0.900, 1.000) & sPSPR & 8878.4 & 5.47e2 &  5.0 & \textbf{4817.9} & 1.60e3 &  2.8 & 6001.5 & 3.19e3 &  3.4 & 8811.9 & 5.31e3 &  5.0 \\ \cdashline{1-14}[0.8pt/2pt]
(0.900, 0.900) & iPSPR & 9090.5 & 5.10e2 &  5.1 & \textbf{4884.2} & 1.52e3 &  2.8 & 5815.1 & 3.03e3 &  3.3 & 8375.5 & 5.04e3 &  4.7 \\
(0.900, 0.900) & sPSPR & 9252.9 & 5.47e2 &  5.2 & \textbf{4982.8} & 1.60e3 &  2.8 & 6075.2 & 3.19e3 &  3.4 & 8880.5 & 5.31e3 &  5.0 \\ \cdashline{1-14}[0.8pt/2pt]
(0.800, 1.000) & iPSPR & 9094.7 & 5.10e2 &  5.2 & \textbf{4887.6} & 1.52e3 &  2.8 & 5823.3 & 3.03e3 &  3.3 & 8405.3 & 5.04e3 &  4.8 \\
(0.800, 1.000) & sPSPR & 9256.5 & 5.47e2 &  5.2 & \textbf{4985.6} & 1.60e3 &  2.8 & 6079.5 & 3.19e3 &  3.5 & 8894.2 & 5.31e3 &  5.0 \\  \cdashline{1-14}[0.8pt/2pt]
(0.809, 0.809) & iPSPR & 9706.1 & 4.86e2 &  5.5 & \textbf{5250.0} & 1.44e3 &  3.0 & 5667.2 & 2.88e3 &  3.2 & 8094.9 & 4.80e3 &  4.6 \\
(0.809, 0.809) & sPSPR & 10058.7 & 5.47e2 &  5.7 & \textbf{5388.7} & 1.60e3 &  3.1 & 6240.7 & 3.19e3 &  3.6 & 8960.9 & 5.31e3 &  5.1 \\  \cdashline{1-14}[0.8pt/2pt]
(0.618, 1.000) & iPSPR & 9722.5 & 4.86e2 &  5.5 & \textbf{5245.5} & 1.44e3 &  3.0 & 5694.4 & 2.88e3 &  3.2 & 8124.6 & 4.80e3 &  4.6 \\
(0.618, 1.000) & sPSPR & 10074.1 & 5.47e2 &  5.7 & \textbf{5392.1} & 1.60e3 &  3.1 & 6261.2 & 3.19e3 &  3.6 & 9007.0 & 5.31e3 &  5.1 \\ \cdashline{1-14}[0.8pt/2pt]
(0.000, 1.618) & iPSPR & 10216.2 & 5.37e2 &  5.8 & \textbf{5511.2} & 1.60e3 &  3.1 & 6602.8 & 3.19e3 &  3.8 & 9583.0 & 5.31e3 &  5.4 \\
(0.000, 1.618) & sPSPR & 10264.2 & 5.47e2 &  5.8 & \textbf{5517.4} & 1.60e3 &  3.1 & 6605.5 & 3.19e3 &  3.8 & 9583.1 & 5.31e3 &  5.4 \\  \cdashline{1-14}[0.8pt/2pt]
(0.000, 1.000) & iPSPR & 13521.7 & 4.05e2 &  7.6 & 7761.5 & 1.20e3 &  4.4 & \textbf{5967.7} & 2.39e3 &  3.4 & 7667.3 & 3.98e3 &  4.3 \\
(0.000, 1.000) & sPSPR & 14513.1 & 5.47e2 &  8.2 & 8041.4 & 1.60e3 &  4.6 & \textbf{7311.9} & 3.19e3 &  4.2 & 9930.7 & 5.31e3 &  5.6 \\  \cdashline{1-14}[0.8pt/2pt]
(0.500, 0.500) & iPSPR & 13308.9 & 4.05e2 &  7.5 & 7724.5 & 1.20e3 &  4.4 & \textbf{5728.7} & 2.39e3 &  3.3 & 7265.2 & 3.98e3 &  4.1 \\
(0.500, 0.500) & sPSPR & 14336.6 & 5.47e2 &  8.1 & 7964.6 & 1.60e3 &  4.5 & \textbf{7097.4} & 3.19e3 &  4.0 & 9540.8 & 5.31e3 &  5.4 \\
\hline
\end{tabular}\label{result-1000}
}
\end{table}

\vskip 3mm

\rev{The numerical results are presented in Tables
\ref{result-1000}-\ref{result-8000}. In the tables, ``iter'' means the averaged iteration numbers, ``r'' denotes the proximal parameter in \eqref{equ:prob:qp:iPSPR}, and  ``t'' means the CPU time in seconds. From the tables, we can make the following observations.}
\revise{First, iPSPR always perform better than the sPSPR. \rev{In particular, for $n = 4000, \beta = 0.15$ and $n = 8000, \beta = 0.07$}, iPSPR can bring about 40\% - 50\% reduction in the number of iterations and the CPU time over the sPSPR. \rev{For $n = 1000$ and $2000$, iPSPR  with large sum $\alpha + \gamma$ performs only slightly better than sPSPR with the same $\alpha$ and $\gamma$. This might be due to that $\beta B^{\Tsf}B$ takes a major part in computing $r$ and the parameter $\tau$ of iPSPR is near to $1$ in this case.} \rev{Second,   iPSPR (resp. sPSPR) with $\alpha = \gamma$ performs slightly better among the choices of $\alpha$ and $\gamma$ with fixed sum.} \rev{Third,} a large $\alpha + \gamma$ sum (near to 2) always performs better \rev{than a small sum} \rev{for a relatively small} $\beta$, \rev{while} a small sum works \rev{better than a large sum for a relatively large} $\beta$. \rev{However, if we choose the best $\beta$ (the corresponding results are marked in bold in each table) for each $\alpha$ and $\gamma$, we can see that iPSPR (resp. sPSPR) a large $\alpha + \gamma$ sum always performs better than iPSPR (resp. sPSPR) with a small sum.}
}

\newpage
\begin{table}[!htp]
	\setlength{\tabcolsep}{2pt}
\centering\small
\caption{\rev{The results for $m = 2000, n = 2000$ over 50 runs. The CPU time is in seconds.}}
\rev{
\begin{tabular}{|cc|rrr|rrr|rrr|rrr|}
 %
 \hline
  &&   \multicolumn{3}{c|}{$\beta =  0.10$}    & \multicolumn{3}{c|}{$\beta = 0.30$} & \multicolumn{3}{c|}{$\beta = 0.50$} &    \multicolumn{3}{c|}{\revise{$\beta = 1.00$}} \\
 \cline{3-14}
 $(\alpha, \gamma)$ & method & iter & $r$ & t  & iter & $r$ & t  & iter & $r$ & t  & iter & $r$ & t  \\   \hline
(0.950, 0.950) & iPSPR & 2192.4 & 2.18e2 &  3.2 & \textbf{1012.0} & 4.43e2 &  1.5 & 1240.6 & 7.22e2 &  1.8 & 2328.8 & 1.43e3 &  3.4 \\
(0.950, 0.950) & sPSPR & 2357.0 & 3.83e2 &  3.4 & \textbf{1121.2} & 5.22e2 &  1.7 & 1309.0 & 7.66e2 &  1.9 & 2417.4 & 1.48e3 &  3.5 \\ \cdashline{1-14}[0.8pt/2pt]
(0.900, 1.000) & iPSPR & 2192.8 & 2.18e2 &  3.2 & \textbf{1012.2} & 4.43e2 &  1.5 & 1240.9 & 7.22e2 &  1.9 & 2329.3 & 1.43e3 &  3.4 \\
(0.900, 1.000) & sPSPR & 2357.3 & 3.83e2 &  3.4 & \textbf{1121.4} & 5.22e2 &  1.7 & 1309.2 & 7.66e2 &  2.0 & 2417.9 & 1.48e3 &  3.5 \\ \cdashline{1-14}[0.8pt/2pt]
(0.900, 0.900) & iPSPR & 2294.6 & 2.17e2 &  3.3 & \textbf{1028.0} & 4.32e2 &  1.5 & 1226.5 & 7.04e2 &  1.8 & 2274.1 & 1.39e3 &  3.3 \\
(0.900, 0.900) & sPSPR & 2474.9 & 3.83e2 &  3.6 & \textbf{1146.1} & 5.22e2 &  1.7 & 1324.5 & 7.66e2 &  2.0 & 2423.7 & 1.48e3 &  3.5 \\ \cdashline{1-14}[0.8pt/2pt]
(0.800, 1.000) & iPSPR & 2295.3 & 2.17e2 &  3.4 & \textbf{1028.4} & 4.32e2 &  1.6 & 1227.1 & 7.04e2 &  1.8 & 2276.1 & 1.39e3 &  3.3 \\
(0.800, 1.000) & sPSPR & 2475.5 & 3.83e2 &  3.6 & \textbf{1146.8} & 5.22e2 &  1.7 & 1325.2 & 7.66e2 &  2.0 & 2425.9 & 1.48e3 &  3.5 \\ \cdashline{1-14}[0.8pt/2pt]
(0.809, 0.809) & iPSPR & 2509.4 & 2.14e2 &  3.6 & \textbf{1074.5} & 4.13e2 &  1.6 & 1201.1 & 6.71e2 &  1.8 & 2170.0 & 1.33e3 &  3.2 \\
(0.809, 0.809) & sPSPR & 2722.4 & 3.83e2 &  3.9 & \textbf{1201.1} & 5.22e2 &  1.8 & 1351.8 & 7.66e2 &  2.0 & 2437.7 & 1.48e3 &  3.5 \\ \cdashline{1-14}[0.8pt/2pt]
(0.618, 1.000) & iPSPR & 2511.7 & 2.14e2 &  3.6 & \textbf{1075.4} & 4.13e2 &  1.6 & 1204.1 & 6.71e2 &  1.8 & 2177.7 & 1.33e3 &  3.2 \\
(0.618, 1.000) & sPSPR & 2724.1 & 3.83e2 &  3.9 & \textbf{1202.5} & 5.22e2 &  1.8 & 1354.7 & 7.66e2 &  2.0 & 2445.8 & 1.48e3 &  3.6 \\ \cdashline{1-14}[0.8pt/2pt]
(0.000, 1.618) & iPSPR & 2541.2 & 2.20e2 &  3.7 & \textbf{1138.7} & 4.53e2 &  1.7 & 1366.0 & 7.40e2 &  2.0 & 2547.4 & 1.47e3 &  3.7 \\
(0.000, 1.618) & sPSPR & 2733.0 & 3.83e2 &  3.9 & \textbf{1228.7} & 5.22e2 &  1.8 & 1406.8 & 7.66e2 &  2.1 & 2571.5 & 1.48e3 &  3.7 \\ \cdashline{1-14}[0.8pt/2pt]
(0.000, 1.000) & iPSPR & 3737.8 & 2.05e2 &  5.4 & 1551.7 & 3.50e2 &  2.3 & \textbf{1242.2} & 5.60e2 &  1.8 & 1910.1 & 1.10e3 &  2.8 \\
(0.000, 1.000) & sPSPR & 4152.4 & 3.83e2 &  5.9 & 1649.0 & 5.22e2 &  2.4 & \textbf{1542.5} & 7.66e2 &  2.3 & 2559.0 & 1.48e3 &  3.7 \\ \cdashline{1-14}[0.8pt/2pt]
(0.500, 0.500) & iPSPR & 3716.2 & 2.05e2 &  5.3 & 1546.0 & 3.50e2 &  2.3 & \textbf{1213.4} & 5.60e2 &  1.8 & 1831.6 & 1.10e3 &  2.7 \\
(0.500, 0.500) & sPSPR & 4137.3 & 3.83e2 &  5.9 & 1637.5 & 5.22e2 &  2.4 & \textbf{1509.1} & 7.66e2 &  2.2 & 2481.4 & 1.48e3 &  3.6 \\
\hline
\end{tabular}\label{result-2000}
}
\end{table}

\begin{table}[!htp]
	\setlength{\tabcolsep}{2pt}
\centering\small
\caption{\rev{The results for $m = 2000, n = 4000$ over 50 runs. The CPU time is in seconds.}}
\rev{
\begin{tabular}{|cc|rrr|rrr|rrr|rrr|}
 %
 \hline
  &&   \multicolumn{3}{c|}{$\beta =  0.08$}    & \multicolumn{3}{c|}{$\beta = 0.15$} & \multicolumn{3}{c|}{$\beta = 0.25$} &    \multicolumn{3}{c|}{$\beta =0.50$} \\
 \cline{3-14}
 $(\alpha, \gamma)$ & method & iter & $r$ & t  & iter & $r$ & t  & iter & $r$ & t  & iter & $r$ & t  \\   \hline
(0.950, 0.950) & iPSPR &  905.6 & 3.71e2 &  3.0 &  \textbf{672.0} & 4.24e2 &  2.3 &  831.3 & 5.65e2 &  2.8 & 1657.8 & 1.06e3 &  5.4 \\
(0.950, 0.950) & sPSPR & 1207.2 & 7.03e2 &  4.0 & \textbf{1103.8} & 7.39e2 &  3.7 & 1233.3 & 8.09e2 &  4.1 & 1813.6 & 1.15e3 &  5.9 \\  \cdashline{1-14}[0.8pt/2pt]
(0.900, 1.000) & iPSPR &  905.7 & 3.71e2 &  3.0 &  \textbf{672.0} & 4.24e2 &  2.3 &  831.4 & 5.65e2 &  2.8 & 1657.9 & 1.06e3 &  5.4 \\
(0.900, 1.000) & sPSPR & 1207.2 & 7.03e2 &  4.0 & \textbf{1103.9} & 7.39e2 &  3.7 & 1233.3 & 8.09e2 &  4.1 & 1813.7 & 1.15e3 &  5.9 \\ \cdashline{1-14}[0.8pt/2pt]
(0.900, 0.900) & iPSPR &  943.5 & 3.70e2 &  3.2 &  \textbf{681.6} & 4.21e2 &  2.3 &  812.2 & 5.54e2 &  2.8 & 1615.6 & 1.03e3 &  5.3 \\
(0.900, 0.900) & sPSPR & 1234.3 & 7.03e2 &  4.1 & \textbf{1103.9} & 7.39e2 &  3.7 & 1231.0 & 8.09e2 &  4.1 & 1815.3 & 1.15e3 &  6.0 \\ \cdashline{1-14}[0.8pt/2pt]
(0.800, 1.000) & iPSPR &  943.7 & 3.70e2 &  3.2 &  \textbf{681.7} & 4.21e2 &  2.3 &  812.3 & 5.54e2 &  2.8 & 1616.2 & 1.03e3 &  5.3 \\
(0.800, 1.000) & sPSPR & 1234.5 & 7.03e2 &  4.1 & \textbf{1104.4} & 7.39e2 &  3.7 & 1231.3 & 8.09e2 &  4.1 & 1816.0 & 1.15e3 &  6.0 \\  \cdashline{1-14}[0.8pt/2pt]
(0.809, 0.809) & iPSPR & 1029.0 & 3.68e2 &  3.4 &  \textbf{706.1} & 4.14e2 &  2.4 &  779.1 & 5.34e2 &  2.7 & 1537.7 & 9.86e2 &  5.1 \\
(0.809, 0.809) & sPSPR & 1294.0 & 7.03e2 &  4.3 & \textbf{1108.9} & 7.39e2 &  3.7 & 1226.0 & 8.09e2 &  4.1 & 1817.5 & 1.15e3 &  6.0 \\  \cdashline{1-14}[0.8pt/2pt]
(0.618, 1.000) & iPSPR & 1029.6 & 3.68e2 &  3.4 &  \textbf{706.8} & 4.14e2 &  2.4 &  780.9 & 5.34e2 &  2.7 & 1540.7 & 9.86e2 &  5.1 \\
(0.618, 1.000) & sPSPR & 1294.5 & 7.03e2 &  4.3 & \textbf{1109.9} & 7.39e2 &  3.7 & 1227.7 & 8.09e2 &  4.1 & 1820.4 & 1.15e3 &  6.0 \\  \cdashline{1-14}[0.8pt/2pt]
(0.000, 1.618) & iPSPR & 1035.6 & 3.73e2 &  3.5 &  \textbf{731.1} & 4.28e2 &  2.5 &  876.7 & 5.77e2 &  3.0 & 1764.3 & 1.09e3 &  5.8 \\
(0.000, 1.618) & sPSPR & 1300.3 & 7.03e2 &  4.3 & \textbf{1125.7} & 7.39e2 &  3.7 & 1256.0 & 8.09e2 &  4.2 & 1876.5 & 1.15e3 &  6.2 \\  \cdashline{1-14}[0.8pt/2pt]
(0.000, 1.000) & iPSPR & 1594.1 & 3.61e2 &  5.2 &  927.6 & 3.95e2 &  3.1 &  \textbf{759.7} & 4.74e2 &  2.6 & 1300.8 & 8.25e2 &  4.3 \\
(0.000, 1.000) & sPSPR & 1727.8 & 7.03e2 &  5.7 & 1232.1 & 7.39e2 &  4.1 & \textbf{1223.3} & 8.09e2 &  4.1 & 1847.1 & 1.15e3 &  6.1 \\  \cdashline{1-14}[0.8pt/2pt]
(0.500, 0.500) & iPSPR & 1589.2 & 3.61e2 &  5.2 &  922.7 & 3.95e2 &  3.1 &  \textbf{747.5} & 4.74e2 &  2.6 & 1264.4 & 8.25e2 &  4.2 \\
(0.500, 0.500) & sPSPR & 1724.1 & 7.03e2 &  5.7 & 1224.8 & 7.39e2 &  4.1 & \textbf{1205.5} & 8.09e2 &  4.0 & 1810.6 & 1.15e3 &  6.0 \\
\hline
\end{tabular}\label{result-4000}
}
\end{table}

\newpage
\begin{table}[!htp]
	\setlength{\tabcolsep}{1.5pt}
\centering\small
\caption{\rev{The results for $m = 2000, n = 8000$ over 50 runs. The CPU time is in seconds.}}
\rev{
\begin{tabular}{|cc|rrr|rrr|rrr|rrr|}
 %
 \hline
  &&   \multicolumn{3}{c|}{$\beta =  0.04$}    & \multicolumn{3}{c|}{$\beta = 0.07$} & \multicolumn{3}{c|}{$\beta = 0.15$} &    \multicolumn{3}{c|}{$\beta =0.30$} \\
 \cline{3-14}
 $(\alpha, \gamma)$ & method & iter & $r$ & t  & iter & $r$ & t  & iter & $r$ & t  & iter & $r$ & t  \\   \hline
(0.950, 0.950) & iPSPR &889.6 & 6.80e2 &  6.6  &  \textbf{759.9} & 6.95e2 &  5.7 &  861.4 & 7.55e2 &  6.4 & 1236.9 & 1.05e3 &  9.1 \\
(0.950, 0.950) & sPSPR & \textbf{1487.6} & 1.34e3 &  10.7 & 1556.1 & 1.36e3 &  11.2 & 1658.1 & 1.40e3 &  11.9 & 1819.2 & 1.52e3 &  13.1 \\ \cdashline{1-14}[0.8pt/2pt]
(0.900, 1.000) & iPSPR &  889.7 & 6.80e2 &  6.6 &  \textbf{760.0} & 6.95e2 &  5.7 &  861.5 & 7.55e2 &  6.4 & 1236.9 & 1.05e3 &  9.1 \\
(0.900, 1.000) & sPSPR & \textbf{1487.6} & 1.34e3 &  10.7 & 1556.2 & 1.36e3 &  11.2 & 1658.1 & 1.40e3 &  11.9 & 1819.2 & 1.52e3 &  13.1 \\ \cdashline{1-14}[0.8pt/2pt]
(0.900, 0.900) & iPSPR & 913.9 & 6.80e2&  6.8 &  \textbf{761.9} & 6.94e2 &  5.7 &  854.8 & 7.51e2 &  6.4 & 1210.7 & 1.03e3 &  8.9 \\
(0.900, 0.900) & sPSPR &  1488.1 & 1.34e3 &  10.7 & \textbf{1550.8} & 1.36e3 &  11.2 & 1655.9 & 1.40e3 &  11.9 & 1819.2 & 1.52e3 &  13.1 \\ \cdashline{1-14}[0.8pt/2pt]
(0.800, 1.000) & iPSPR &  913.9 & 6.80e2 &  6.8 &  \textbf{762.0} & 6.94e2 &  5.8 &  854.9 & 7.51e2 &  6.4 & 1211.0 & 1.03e3 &  9.0 \\
(0.800, 1.000) & sPSPR & 1488.3 & 1.34e3 &  10.7 & \textbf{1550.9} & 1.36e3 &  11.3 & 1656.0 & 1.40e3 &  12.0 & 1819.6 & 1.52e3 &  13.2 \\ \cdashline{1-14}[0.8pt/2pt]
(0.809, 0.809) & iPSPR & 968.8 & 6.79e2 &  7.1 &  \textbf{770.7} & 6.93e2 &  5.8 &  842.1 & 7.44e2 &  6.3 & 1164.0 & 9.89e2 &  8.7 \\
(0.809, 0.809) & sPSPR & 1497.5 & 1.34e3 &  10.8 & \textbf{1539.4} & 1.36e3 &  11.2 & 1650.8 & 1.40e3 &  12.0 & 1819.0 & 1.52e3 &  13.2 \\ \cdashline{1-14}[0.8pt/2pt]
(0.618, 1.000) & iPSPR &  969.2 & 6.79e2 &  7.1 &  \textbf{771.2} & 6.93e2 &  5.8 &  842.9 & 7.44e2 &  6.3 & 1165.4 & 9.89e2 &  8.6 \\
(0.618, 1.000) & sPSPR & 1497.7 & 1.34e3 &  10.8 & \textbf{1539.8} & 1.36e3 &  11.2 & 1651.7 & 1.40e3 &  12.0 & 1820.6 & 1.52e3 &  13.2 \\ \cdashline{1-14}[0.8pt/2pt]
(0.000, 1.618) & iPSPR & 972.6 & 6.81e2 &  7.1 &  \textbf{779.9} & 6.96e2 &  5.9 &  873.9 & 7.59e2 &  6.6 & 1289.3 & 1.07e3 &  9.5 \\
(0.000, 1.618) & sPSPR & 1501.2 & 1.34e3 &  10.8 & \textbf{1546.2} & 1.36e3 &  11.2 & 1664.1 & 1.40e3 &  12.0 & 1845.1 & 1.52e3 &  13.4 \\ \cdashline{1-14}[0.8pt/2pt]
(0.000, 1.000) & iPSPR & 1406.3 & 6.76e2 &  10.2 &  922.8 & 6.87e2 &  6.9 &  \textbf{799.3} & 7.25e2 &  6.1 & 1030.0 & 8.76e2 &  7.7 \\
(0.000, 1.000) & sPSPR & 1732.8 & 1.34e3 &  12.4 & \textbf{1500.3} & 1.36e3 &  11.0 & 1625.9 & 1.40e3 &  11.8 & 1825.7 & 1.52e3 &  13.3 \\ \cdashline{1-14}[0.8pt/2pt]
(0.500, 0.500) & iPSPR & 1403.0 & 6.76e2 &  10.1 &  919.9 & 6.87e2 &  6.9 &  \textbf{790.6} & 7.25e2 &  6.0 & 1014.0 & 8.76e2 &  7.6 \\
(0.500, 0.500) & sPSPR & 1730.7 & 1.34e3 &  12.4  & \textbf{1496.1} & 1.36e3 &  10.9 & 1617.5 & 1.40e3 &  11.8 & 1809.7 & 1.52e3 &  13.1 \\
\hline
\end{tabular}\label{result-8000}
}
\end{table}

\section{Conclusions}\label{section:conclusions}

In this paper, we proposed a modification of the Peaceman-Rachford splitting method by introducing two different parameters $\alpha$ and $\gamma$ in updating the dual variable, and by introducing \revise{indefinite} proximal terms to the subproblems in updating the primal variables. We established the relationship between the two parameters $\alpha$ and $\gamma$ and proved the global convergence of the algorithm under some mild \rev{requirements on}  the proximal matrices $S$ and $T$. \revise{Moreover, we provided a specific construction of the proximal matrix $T$ and discussed the detailed performance for the variants parameters $\alpha$ and $\gamma$ which can unify several existing results. We also analyzed the $o(1/t)$ sublinear rate convergence in the nonergodic sense.} Finally, we reported some \rev{preliminary} numerical results, indicating the efficiency of the proposed algorithm.

Note that the parameters $\alpha$ and $\gamma$ are essential to the efficiency of the algorithm, which should be variable along with the iteration. Allowing the parameter $\alpha$ and $\gamma$ varying with the process of the iterate may give us the freedom of choosing them in a self-adaptive manner. Such suitable updating rules are among our future research tasks. Besides, an approximate version of the proposed iPSPR with practical accuracy criteria is also our future research topic.

\vskip 2mm
{\noindent {\bf Acknowledgments.}
The authors thank the referee for their valuable comments which improves the early version of the article.
\revis{This paper is an improved version of our earlier paper \cite{gu2015semi}.}
Part of this work was done during Y. Gu's Master study in Nanjing Normal University and PhD study in Kyoto University. Y. Gu was partially supported by the Natural Science Foundation of Jiangsu Province (Grant No. BK20210267). B. Jiang was supported by the National Natural Science Foundation of China (11971239) and the Natural Science Foundation of the Higher Education
Institutions of Jiangsu Province (21KJA110002). D. Han was supported by the National Natural Science Foundation of China (12131004, 11625105).
}

\end{document}